\documentclass[12pt, a4paper]{amsart}
\usepackage{amscd, amssymb, pb-diagram}
\usepackage{mathtools}
\usepackage{mathrsfs}
\usepackage{enumerate}
\usepackage[shortlabels]{enumitem}
\usepackage{amsmath}
\usepackage[margin=2.9cm]{geometry}
\usepackage{amscd, amssymb, pb-diagram}
\usepackage{tikz-cd}
\usepackage{amsthm} 
\usepackage[nospace,noadjust]{cite}
\usepackage{lmodern}
\usepackage{graphicx}
\graphicspath{{images/}{../images/}}
\usepackage[toc,page]{appendix}
\usepackage[normalem]{ulem}
\usepackage{emptypage}
\usepackage{hyperref}
\usepackage{subfiles}
\usepackage{datetime}

\usepackage{tensor}
\usepackage{wrapfig}
\usepackage{tikz}
\usepackage{multicol}
\usepackage[autostyle]{csquotes}
\DeclarePairedDelimiter\ceil{\lceil}{\rceil}

\allowdisplaybreaks

\def\Dom{\operatorname{Dom}}

\def\ker{\operatorname{Ker}}

\def\sup{\operatorname{sup}}

\def\Im{\operatorname{Im}}

\def\diam{\operatorname{diam}}
\def\Lip{\operatorname{Lip}}

\def\D{\operatorname{Din}}
\def\W{\operatorname{Dom\, \mathcal{E}}}
\def\M{\operatorname{M}}
\def\m{\operatorname{m}}
\def\T{\operatorname{T}}
\def\I{\operatorname{I}}
\def\J{\operatorname{J}}
\def\Li{\operatorname{Li}}
\def\K{\operatorname{K}}
\def\R{\operatorname{R}}
\def\A{\operatorname{A}}
\def\Cucb{\operatorname{\textup{C}_{ucb}}}
\def\h{\operatorname{h}}

\def\Xint#1{\mathchoice
   {\XXint\displaystyle\textstyle{#1}}%
   {\XXint\textstyle\scriptstyle{#1}}%
   {\XXint\scriptstyle\scriptscriptstyle{#1}}%
   {\XXint\scriptscriptstyle\scriptscriptstyle{#1}}%
   \!\int}
\def\XXint#1#2#3{{\setbox0=\hbox{$#1{#2#3}{\int}$}
     \vcenter{\hbox{$#2#3$}}\kern-.5\wd0}}

\def\dashint{\Xint-}

\newcommand{\Hol}{\textup{H\"ol}}

\newcommand{\Ld}{\Delta}

\newcommand{\cB}{\overline{B}}

\newcommand{\E}{\mathcal{E}}
\newcommand*{\dd}{\mathop{}\!\mathrm{d}}
\newcommand{\Q}{\mathcal{Q}}

\makeatletter
\def\Ddots{\mathinner{\mkern1mu\raise\p@
\vbox{\kern7\p@\hbox{.}}\mkern2mu
\raise4\p@\hbox{.}\mkern2mu\raise7\p@\hbox{.}\mkern1mu}}
\makeatother

\numberwithin{equation}{section}

\tikzstyle{vertex}=[circle]
\tikzstyle{goto}=[->,shorten >=1pt,>=stealth,semithick]

\newtheorem{thm}{Theorem}[section]
\newtheorem{cor}[thm]{Corollary}
\newtheorem{lemma}[thm]{Lemma}
\newtheorem{prop}[thm]{Proposition}

\newtheorem{mainthm}{Theorem}

\theoremstyle{definition}
\newtheorem{definition}[thm]{Definition}

\theoremstyle{remark}
\newtheorem{remark}[thm]{Remark}

\newtheorem*{Acknowledgements}{Acknowledgements}

\makeatletter
\@namedef{subjclassname@2020}{%
  \textup{2020} Mathematics Subject Classification}
\makeatother
\begin{document}

\begin{abstract}
We introduce the logarithmic analogue of the Laplace-Beltrami operator on Ahlfors regular metric-measure spaces. This operator is intrinsically defined with spectral properties analogous to those of elliptic pseudo-differential operators on Riemannian manifolds. Specifically, its heat semigroup consists of compact operators which are trace-class after some critical point in time. Moreover, its domain is a Banach module over the Dini continuous functions and every H\"older continuous function is a smooth vector. Finally, the operator is compatible, in the sense of noncommutative geometry, with the action of a large class of non-isometric homeomorphisms. 
\end{abstract}

\title[The logarithmic Dirichlet Laplacian on Ahlfors regular spaces]{The logarithmic Dirichlet Laplacian on\\ Ahlfors regular spaces}
\author[D.M. Gerontogiannis]{Dimitris Michail Gerontogiannis}
\address{\scriptsize Dimitris Michail Gerontogiannis, Bram Mesland, Leiden University, Niels Bohrweg 1, 2333 CA Leiden, The Netherlands}
\email{d.m.gerontogiannis@math.leidenuniv.nl}

\author[B. Mesland]{Bram Mesland}
\email{b.mesland@math.leidenuniv.nl}

\maketitle
\tableofcontents
\section*{Introduction}

The Laplace-Beltrami operator is a fundamental tool in the study of compact Riemannian manifolds. Over the years, the study of analogous operators on general metric measure spaces has generated a vast literature. Amongst the most notable works are those of Cheeger \cite{Ch} for certain length spaces with doubling measures, of Kigami \cite{Ki} for p.c.f self-similar sets and of Sturm \cite{St} for spaces with the measure contraction property.

Here we pursue a different path and study a non-local, logarithmic analogue of the Laplace-Beltrami operator that exhibits remarkable properties, in particular in its  compatibility with dynamics. We focus on metric-measure spaces $(X,d,\mu)$ with $\mu$ being a finite Ahlfors $\delta$-regular measure. Examples are compact Riemannian manifolds, several fractals \cite{Falc}, self-similar Smale spaces \cite{Ge} and limit sets of hyperbolic isometry groups \cite{Coorn}. For compact Riemannian manifolds, our operator is a bounded perturbation of the logarithm of the Laplace-Beltrami operator. In general, it is the generator $\Delta$ of the Dirichlet form 
\begin{equation}\label{eq: intro0}
\E(f,g):=\frac{1}{2}\int_X\int_X \frac{(f(x)-f(y))(g(x)-g(y))}{d(x,y)^{\delta}}\dd\mu(y)\dd\mu(x),
\end{equation}
with $\W = \{f\in L^2(X,\mu):\E (f,f) <\infty \}$ in the real Hilbert space $L^2(X,\mu)$. We call $\Delta$ the \textit{logarithmic Dirichlet Laplacian}, due to the logarithmic singularity of $\E$ (see Lemma \ref{lem:Ahlfors_estimates}) which forces  the logarithmic growth of its spectrum. 

At first glance, $\Delta$ resembles the fractional Dirichlet Laplacian $\Delta_{\alpha}$, for $0<\alpha <1$. The latter is the generator of the Dirichlet form given by (\ref{eq: intro0}), but with singular kernel $d(x,y)^{-(\delta +\alpha)}$ instead of $d(x,y)^{-\delta}$. This form has been studied in terms of its H\"older continuous heat kernel and plays an important role in the study of fractional Sobolev spaces, see Chen-Kumagai-Wang \cite{CKW}, Grigor'yan-Hu-Lau \cite{GHL} and Nahmod's Calder\'on-Zygmund theoretic approach \cite{Nah}.

However, $\Delta_{\alpha}$ differs quite drastically from $\Delta$ due to the logarithmic behaviour of the latter, which does not allow $\Delta$ to have such a regular heat kernel. The trade-off is the compatibility of $\Delta$ with non-isometric group actions on the metric space $(X,d)$. This feature is not enjoyed by $\Delta_{\alpha}$, as we showcase in Subsection \ref{sec:CBA}. It is exactly this property that allows us to incorporate actions of Kleinian groups on their limit sets in Connes' noncommutative geometry programme \cite{Connes}. Although in this paper we refrain from expanding on this programme, we note that our approach to analysing $\Delta$ has been dictated by it.

At this point we mention that, in order to be consistent with the literature on Dirichlet forms, we focus on real Hilbert spaces, unless stated otherwise. Nevertheless, our results extend mutatis mutandis to complex Hilbert spaces by complexifying $\Delta$, see Section \ref{sec:logLap}.

After formally defining $\Delta$ in Section \ref{sec:logLap}, we prove our first theorem about its spectral properties in Section \ref{sec:resolvent}. The logarithmic growth of the eigenvalues is a consequence of the logarithmic singularity in (\ref{eq: intro0}) and the fact that $(X,d)$ is a doubling metric space, as the measure $\mu$ is Ahlfors regular. 

\begin{mainthm}[see Theorem {\ref{theorem:resolvent}}]\label{thm:mainA}
The operator $\Delta$ has compact resolvent and its eigenvalues grow logarithmically fast. In particular, there is some $t_0>0$ so that the heat operators $e^{-t\Delta}$ are trace-class, when $t>t_0$.
\end{mainthm}

An intriguing aspect of the threshold $t_0>0$ is that it seems to be a topological invariant of $(X,d)$ that scales inversely with $\Delta$, see the examples in Section \ref{sec:Examples}. It would be interesting to understand the minimal value of $t_0$ from a Weyl Law point of view. In that way we would know that, in general, if $t\leq t_0$, the operator $e^{-t\Delta}$ cannot be trace-class and hence for $t\leq t_0/2$ not even Hilbert-Schmidt, as we observe in the examples. This has implications for the regularity of the heat kernel of $\Delta$, but can also be used to study KMS states for dynamical systems on Ahlfors regular spaces, see \cite{GRU}.

In Section \ref{sec:domain} we analyse the domain $\Dom\Delta$ and establish a close relation to the Dini continuous functions on the underlying metric space $(X,d)$. The latter form a Banach algebra $\D(X,d)$ containing all H\"older continuous functions. Specifically, we first show that $\W$ is a Banach module over $\D(X,d)$. Then, by studying the commutator of $\Ld$ with the multiplication operators of Dini continuous functions, we derive the same Banach module structure for $\Dom\Delta$. In particular we prove the following.

\begin{mainthm}[see Theorem {\ref{theorem:domain_module}}]\label{thm:mainB}
Let $h$ be a Dini continuous function on $(X,d)$. Then the multiplication operator $\m_h$ on $L^2(X,\mu)$ satisfies $\m_h:\Dom \Delta \to \Dom \Delta$ and the commutator $[\Delta , \m_h]$ extends to a bounded operator on $L^2(X,\mu)$.
\end{mainthm}

Theorem \ref{thm:mainB} implies that all Dini continuous functions are in $\Dom\Delta$, see Corollary \ref{cor:Din_in_Dom}. In Subsection \ref{sec:integralrep} we then derive an integral representation of $\Delta$ on Dini continuous functions $f$, namely 
\begin{equation}\label{eq:intro3}
\Delta f(x)=\int_X\frac{f(x)-f(y)}{d(x,y)^{\delta}} \dd\mu(y).
\end{equation}
In Subsection \ref{sec:smooth} we obtain a distinctive property of $\Delta$: every H\"older continuous function is a smooth vector for $\Delta$, a property that cannot hold for any $\Delta_{\alpha}$.

The integral formula (\ref{eq:intro3}) (for complex-valued functions) has been employed before in cases where it could be well understood as a densely defined positive essentially self-adjoint operator on the Lipschitz continuous functions. Specifically, the second-named author used it on full shift spaces to study the Cuntz $C^*$-algebras in joint work with Goffeng \cite{GM}, and on spheres to study Bianchi groups in joint work with Şeng\"un \cite{MS}. In addition, Goffeng-Usachev \cite{GU} studied the operator on compact Riemannian manifolds and showed that it is a bounded perturbation of the logarithm of the Laplace-Beltrami operator.

Furthermore, Chen-Weth \cite{CW} recently studied the Poisson problem on bounded domains in $\mathbb R^N$ involving the logarithmic Laplacian, which is a singular operator with symbol $2\log |\xi|$. The integral representation of this operator on $\mathbb R^N$ resembles (\ref{eq:intro3}), except for the fact that it is localised around $x\in \mathbb R^N$ so that the integral over $\mathbb R^N$ converges for nice functions. An interesting aspect of it has been observed by Chen-V\'eron \cite{CV}, asserting that if the volume of the bounded domain is large enough, the first eigenvalue of the operator will be negative, whereas the eigenvalues become positive eventually. This does not happen in our situation and the eigenvalues of $\Delta$ are always positive. 

We conclude the paper with several examples in Section \ref{sec:Examples}. First of all, Theorems \ref{thm:mainA} and \ref{thm:mainB} immediately imply that every Ahlfors regular metric-measure space can be viewed as a noncommutative manifold: the Banach algebra of Dini continuous functions, the Hilbert space $L^{2}(X,\mu)$ and the operator $\Delta$ assemble to a so-called spectral triple in the sense of noncommutative geometry. 

Subsequently we showcase that $\Delta$ generalises several well-known operators. Specifically, apart from being a bounded perturbation of the logarithm of the Laplace-Beltrami operator on compact Riemannian manifolds, $\Delta$ generalises the Julien-Putnam \cite{JP} operator on the Cantor sets of full shifts with the Bernoulli measure. Also, it acts as a scalar on the Peter-Weyl $L^2$-decomposition in the case of compact groups with an Ahlfors regular Haar measure, in which case $\Delta$ can be viewed as a Casimir-type operator. Moreover, on closed intervals, $\Delta$ is diagonalised by the Legendre polynomials with eigenvalues the harmonic numbers. The latter has already been observed by Tuck \cite{T} in 1964 and is a remarkable fact as the interval is a manifold with boundary.

Finally, we prove that for a Kleinian group $\Gamma$, either of the first kind, or convex cocompact of the second kind, the action of $\Gamma$ on its limit set $\Lambda_{\Gamma}$ equipped with the Patterson-Sullivan measure $\mu$, is differentiable with respect to $\Delta$. That is, except for $\Delta$ having bounded commutators with the multiplication operators of Dini continuous functions on $\Lambda_{\Gamma}$, every $\gamma \in \Gamma$ viewed as a unitary operator on $L^2(\Lambda_{\Gamma},\mu)$ has a bounded commutator with $\Delta$. The latter cannot hold for the fractional Dirichlet Laplacian $\Delta_{\alpha}$ due to a tracial obstruction of Connes in noncommutative geometry, as we discuss in Remark \ref{rem:CBA}.

\section{Preliminaries}\label{sec_Prelim}
First we fix some notation that will often be used. Specifically, if $F$ and $G$ are real valued functions on some parameter space $Z$, we write $F \lesssim G$ whenever there is a constant $C>0$ such that for all $z\in Z$ we have $F(z) \leq C G(z)$. Similarly, we define the symbol $\gtrsim$ and write $F\simeq G$ if $F \lesssim G$ and $F \gtrsim G$. Further, given a metric-measure space $(X,d,\mu)$ and a measurable subset $Y\subset X$ as well as a measurable function $f:Y\to \mathbb{R}$ we denote by
\[\dashint_{Y}f(y)\dd\mu(y):=\frac{1}{\mu(Y)}\int_{Y} f(y)\dd\mu(y),\]
the average of $f$ over $Y$.

\subsection{Ahlfors regular metric-measure spaces} 
By a \emph{metric-measure space} we mean a triple $(X,d,\mu)$ such that $(X,d)$ is a locally compact second-countable metric space equipped with a finite Borel measure $\mu$. 
\begin{definition}[{\cite{MT}}] Let $\delta >0$. A metric-measure space $(X,d,\mu)$ is \textit{Ahlfors} $\delta$\textit{-regular} if there is a constant $C\geq 1$, so that for every $x\in X$ and $0\leq r<\diam X$ it holds that 
\begin{equation*}
C^{-1}r^{\delta}\leq \mu(\cB(x,r))\leq Cr^{\delta}.
\end{equation*}
\end{definition}

We note that $(X,d,\mu)$ can be Ahlfors $\delta$-regular for a unique $\delta>0$. In that case, $\mu$ is comparable to the $\delta$-dimensional Hausdorff measure and $\delta$ coincides with the Hausdorff, box-counting and Assouad dimensions of $(X,d)$, see \cite{Falc, MT}. Moreover, the finiteness of $\mu$ then implies that the metric space $(X,d)$ is totally bounded, hence $\diam(X)<\infty$. 

The next lemma contains useful regularity estimates which are immediately derived from decomposing the integrals over annuli. For the fourth estimate the assumption $\mu(X)<\infty$ is important, for if $\mu(X)=\infty$, the integral will be infinite too. This is the main reason why we consider the logarithmic Dirichlet Laplacian for finite measures, while its fractional counterpart also makes sense for infinite measures.

\begin{lemma}\label{lem:Ahlfors_estimates}
Let $0<r\leq \diam(X)$ and $s>0$. Then, 
\begin{enumerate}[(1)]
\item for every $x\in X$ it holds that, $$\int_{B(x,r)} \frac{1}{d(x,y)^{\delta -s}} \dd\mu(y)\leq Ce^{\delta+s}(e^s-1)^{-1}r^s;$$ 
\item for every $x\in X$ it holds that, $$\int_{X\setminus B(x,r)} \frac{1}{d(x,y)^{\delta+s}}\dd\mu(y)\leq Ce^{\delta+s}(e^s-1)^{-1}r^{-s};$$
\item it holds that, $$\int_{B(x,r)\setminus B(x,e^{-1}r)}\frac{1}{d(x,y)^{\delta}}\dd\mu(y)\lesssim 1;$$
\item there is $0<c<1$ so that, if in addition $0<r<c$, for every $x\in X$ we have, $$\int_{X\setminus B(x,r)} \frac{1}{d(x,y)^{\delta}}\dd\mu(y)\simeq \log(r^{-1}).$$
\end{enumerate}
\end{lemma}

\begin{proof}
For the first three inequalities see \cite[Lemma 2.5]{GSV}. For the fourth estimate we write $$X\setminus B(x,r)=\bigcup_{n=0}^{\ceil{\log(r^{-1} \textnormal{diam}(X))}}B(x,e^{n+1}r)\setminus B(x,e^{n}r)$$ and using the Ahlfors regularity we obtain $$\int_{X\setminus B(x,r)} \frac{1}{d(x,y)^{\delta}}\dd\mu(y)\simeq \ceil{\log(r^{-1} \textnormal{diam}(X))}.$$ The desired estimate follows by finding small enough $0<c<1$ so that if $0<r<c$, then it holds $\ceil{\log(r^{-1} \textnormal{diam}(X))}\simeq \log(r^{-1})$. 
\end{proof}

\subsection{Dirichlet forms}\label{sec:Dir}
For the theory of Dirichlet forms we refer to \cite{FOT}. Here we briefly recall the main facts needed for the present paper. We note that all Hilbert spaces are assumed to be real for consistency with the literature, unless stated otherwise.

Let $H$ be a Hilbert space and $\Dom\Q \subset H$ be a linear subspace that is the domain of a bilinear form $\Q:\Dom\Q \times\Dom\Q \to \mathbb{R}$. Then, $\Dom\Q$ carries the following inner product and induced norm
\begin{equation}
\label{eq: form-inner-product}
\langle f,g\rangle_{\Q}:=\langle f,g\rangle_{H}+\Q(f,g), \quad \|f\|_{\Q}^{2}:=\langle f,f\rangle_{\Q}.
\end{equation}
The form $\Q$ is \emph{densely defined} if $\Dom\Q\subset H$ is dense and \emph{closed} if $\Dom\Q$ is a Hilbert space with respect to the inner product \eqref{eq: form-inner-product}.

Given a densely defined closed form, the inner product on $H$ defines a bilinear form
\[H\times\Dom\Q\to \mathbb{R},\quad (f,g)\mapsto \langle f,g\rangle_{H},\]
so that 
\begin{align*}
|\langle f,g\rangle_H|\leq \|f\|_{H}\|g\|_{\Q}.
\end{align*}
Thus $f\in H$ defines a continuous linear functional on the Hilbert space $\Dom \Q$ and hence there is a bounded operator $R_{\Q}:H\to \Dom\Q$ determined by
\[\langle R_{\Q}f,g\rangle_{\Q}=\langle f,g\rangle_{H}.\]
Using the contractive embedding $\Dom\Q\hookrightarrow H$, we view $R_{\Q}$ as an operator from $H$ to $H$. Further, the bounded operator $R_{\Q}$ is injective, self-adjoint and $\|R_{\Q}\|\leq 1$. In particular, it has dense range.

\begin{thm}
\label{thm: form-operator-equivalence}
There is a one-to-one correspondence between densely defined closed bilinear forms $\Q$ on $H$ and positive self-adjoint operators $\A_{\Q}:\Dom\A_{\Q}\to H$ with $\Dom \A_{\Q}:=R_{\Q}H$, defined for $f\in H$ as $$\A_{\Q} R_{\Q}f:=f-R_{\Q}f.$$ In fact, we have $R_{\Q}=(1+\A_{\Q})^{-1}$, $\Dom \Q=\Dom\A_{\Q}^{1/2}$ and $\Q(f,g)=\langle \A_{\Q}^{1/2}f,\A_{\Q}^{1/2}g\rangle_{H}$.
\end{thm}
Here we are interested in the Hilbert space $H=L^{2}(X,\mu)$. For functions $f,g\in L^{2}(X,\mu)$ we say that $g$ is a \emph{normal contraction} of $f$, written $g\prec f$, if for $\mu$-almost all $x,y\in X$ we have
\begin{equation*}
|g(x)|\leq |f(x)|,\quad |g(x)-g(y)|\leq |f(x)-f(y)|.
\end{equation*}
\begin{definition} A densely defined closed bilinear form $\Q$ on $L^{2}(X,\mu)$ is a \emph{Dirichlet form} if 
\begin{equation}
\label{eq: normalcontractionproperty}
f\in\Dom \Q,\,\, g\prec f \Rightarrow g\in\Q,\,\, \Q(g,g)\leq \Q(f,f).
\end{equation}
\end{definition}
An equivalent formulation of the normal contraction property \eqref{eq: normalcontractionproperty} is known as the Markov property. Dirichlet forms and their associated operators have a rich theory. Although we will not use the normal contraction property in any essential way in the present paper, the closed bilinear forms we consider are Dirichlet forms.

\section{The logarithmic Dirichlet Laplacian}\label{sec:logLap}
Let $(X,d,\mu)$ be an Ahlfors $\delta$-regular metric-measure space and consider the real vector space
\begin{equation}
\label{eq: Sobolevspace}
\W:=\left\{f\in L^2(X,\mu):\int_X\int_X\frac{|f(x)-f(y)|^2}{d(x,y)^{\delta}}\dd\mu(y)\dd\mu(x)<\infty\right\},
\end{equation}
and the bilinear form $\E:\W\times \W \to \mathbb R$ given by
\begin{equation}
\label{eq: formdef}
\E(f,g):=\frac{1}{2}\int_X\int_X \frac{(f(x)-f(y))(g(x)-g(y))}{d(x,y)^{\delta}}\dd\mu(y)\dd\mu(x).
\end{equation}
The integral defining $\E$ is to be interpreted as an integral over $X\times X\setminus D$, with $$D:=\{(x,x):x\in X\}\subset X\times X,\quad (\mu\times \mu)(D)=0,$$ the diagonal. We denote by $\Hol_{\alpha}(X,d)$ the space of H{\"o}lder continuous functions of exponent $0<\alpha \leq 1$ on $X$. Since $\mu(X)<\infty$, we have that $\Hol_{\alpha}(X,d)$ is a dense subset of $L^{2}(X,\mu)$. For $f\in\Hol_{\alpha}(X,d)$ we denote by $\Hol_{\alpha}(f)$ the $\emph{H{\"o}lder constant}$ of $f$, which is the smallest positive real number such that for all $x,y\in X$ we have  $|f(x)-f(y)|\leq \Hol_{\alpha}(f)d(x,y)^{\alpha}$.
\begin{lemma}\label{lem:Hol_dense} For each $0<\alpha \leq 1$ we have $\Hol_{\alpha}(X,d)\subset \W$. In particular the bilinear form $\E$ is densely defined on $L^{2}(X,\mu)$.
\end{lemma}
\begin{proof} For $f\in \Hol_{\alpha}(X,d)$ Lemma \ref{lem:Ahlfors_estimates} gives,
\begin{align*}
\int_{X}\int_{X}\frac{|f(x)-f(y)|^2}{d(x,y)^{\delta}}\dd\mu(y)\dd\mu(x)&\leq \Hol_{\alpha}(f)^{2}\int_{X}\int_{X}\frac{1}{d(x,y)^{\delta-2\alpha}}\dd\mu(y)\dd\mu(x)\\
&\leq \Hol_{\alpha}(f)^{2}\mu(X)Ce^{\delta+2\alpha}(e^{2\alpha}-1)^{-1}\textnormal{diam}(X)^{2\alpha},
\end{align*}
and thus the assertion follows.
\end{proof}

\begin{remark}
In the proof of Lemma \ref{lem:Hol_dense} we have used the fact that $\mu(X)<\infty$. If $\mu(X)=\infty$, the form $\E$ typically will not be densely defined, for instance in the case of $X=\mathbb R^N$. In \cite{CW} the authors define a logarithmic Laplacian on $\mathbb R^N$ via a certain localisation procedure which yields an operator that is different from the one considered in the present paper.
\end{remark}

\begin{prop} The bilinear form $\mathcal{E}$ defined in \eqref{eq: formdef} is a Dirichlet form. 
\end{prop}
\begin{proof}
From \cite[Example 1.2.4]{FOT} we have that $\E$ is closed, hence $\W$ is a Hilbert space. Moreover, the normal contraction property can proved directly using the definition of $\Dom\E$: 
if $f\in\Dom \E$ and $g\prec f$ then $g\in L^2(X,\mu)$ and the inequality
\[\frac{|g(x)-g(y)|}{d(x,y)^{\delta}}\leq \frac{|f(x)-f(y)|}{d(x,y)^{\delta}},\]
holds almost everywhere on $X\times X\setminus D$.
\end{proof}

\begin{definition}The \textit{logarithmic Dirichlet Laplacian} on $(X,d,\mu)$ is the self-adjoint operator $\Delta$ associated to the Dirichlet form \eqref{eq: formdef} via Theorem \ref{thm: form-operator-equivalence}.
\end{definition}
We write $L^{2}_{\mathbb{C}}(X,\mu)$ for the Hilbert space of complex valued $L^{2}$-functions on $X$. The space $\W_{\mathbb{C}}$ is defined as in Equation \eqref{eq: Sobolevspace}, using functions in $L^{2}_{\mathbb{C}}(X,\mu)$. The relevant bilinear form is then
\[\mathcal{E}_{\mathbb{C}}(f,g):=\frac{1}{2}\int_X\int_X \frac{\overline{(f(x)-f(y))}(g(x)-g(y))}{d(x,y)^{\delta}}\dd\mu(y)\dd\mu(x).\]
Under the identification $L^{2}_{\mathbb{C}}(X,\mu)\cong L^{2}(X,\mu)\otimes \mathbb{C}$, the operator $\Delta_{\mathbb{C}}$ associated to the Dirichlet form $\mathcal{E}_{\mathbb{C}}$ corresponds to $\Delta\otimes 1$. In the remainder of the paper, we work with the real operator $\Delta$. The results of Sections \ref{sec:resolvent} and \ref{sec:domain} imply the analogous results for $\Delta_{\mathbb{C}}$.
\section{Analysis of the resolvent}\label{sec:resolvent}
In this section we show that $\Ld$ has compact resolvent and hence discrete spectrum. Also, we prove that its eigenvalues grow logarithmically fast. This behaviour has implications for the heat semigroup $\{e^{-t\Ld}\}_{t>0}$, which consists of trace class operators when $t$ is large enough. In noncommutative geometry this property is known as $\Li$\emph{-summability} \cite{Connestheta, GRU}. 

\subsection{Compactness}
We prove the compactness of $(1+\Delta)^{-1}$ by showing that the embedding $\W\hookrightarrow L^2(X,\mu)$ is compact. This suffices by the following general observation.
\begin{lemma}Let $\mathcal{Q}$ be a densely defined closed bilinear form and $\A$ the associated operator. The embedding $\Dom \A\hookrightarrow \Dom \Q$ is continuous.
\end{lemma}
\begin{proof} Let $f\in\Dom \A$. Then
\begin{align*}
\langle f,f\rangle_{\mathcal{Q}} =\langle f,f\rangle_{L^2}+\frac{1}{2}\left(\langle \A f,f\rangle_{L^2}+\langle f,\A f\rangle_{L^2}\right)\leq \frac{3}{2}\left(\langle f,f\rangle_{L^2}+\langle \A f,\A f\rangle_{L^2}\right)=\frac{3}{2}\langle f,f\rangle_{\A},
\end{align*}
which proves that the embedding of Hilbert spaces $\Dom\A\hookrightarrow\Dom \mathcal{Q}$ is continuous.
\end{proof}

\begin{prop}\label{prop:compact_resolvent}
The logarithmic Dirichlet Laplacian has compact resolvent.
\end{prop}

\begin{proof}
Let $K\subset \W$ be bounded and we claim that $K$ is totally bounded in $L^2(X,\mu)$. To this end, fix an arbitrary $0<r<c$ as in part (4) of Lemma \ref{lem:Ahlfors_estimates} and consider the truncated kernel $t_r:X\times X \to \mathbb R$ given by
\begin{equation*}
t_r(x,y)=
\begin{cases}
d(x,y)^{-\delta}, & \text{if}\,\, d(x,y)\geq r \\
0, & \text{otherwise}
\end{cases}.
\end{equation*}
Following again part (4) of Lemma \ref{lem:Ahlfors_estimates}, for every $x\in X$ one has that $\|t_r(x,\cdot )\|_{L^1}\simeq \log(r^{-1}).$
We now normalise $t_r$ as follows. Define the kernel 
\begin{equation*}
T_r:X\times X\to \mathbb R,\quad T_{r}(x,y):=\frac{t_r(x,y)}{\|t_r(x,\cdot )\|_{L^1}},
\end{equation*}
so that $\|T_r(x,\cdot)\|_{L^1}=1.$ Let now $\T_r:L^2(X,\mu)\to L^2(X,\mu)$ be the compact operator defined as
\begin{equation*}
\T_rf(x)=\int_X T_r(x,y)f(y)\dd \mu(y).
\end{equation*}
Since $K$ is bounded in $\W$, it is also bounded in $L^2(X,\mu)$ and thus $\T_r(K)$ is totally bounded in $L^2(X,\mu)$. Further, for every $f\in \W$ we have that
\begin{align*}
\|f-\T_rf\|_{L^2}^2&=\int_X \left| \int_X(f(x)-f(y))T_r(x,y)^{\frac{1}{2}}T_r(x,y)^{\frac{1}{2}}\dd\mu(y)\right|^2\dd\mu(x)\\
&\leq \int_X\int_X |f(x)-f(y)|^2 T_r(x,y)\dd\mu(y) \dd\mu(x) \\
&\simeq \frac{1}{\log(r^{-1})}\int_X\int_X |f(x)-f(y)|^2 t_r(x,y)\dd\mu(y) \dd\mu(x)\\
&\lesssim \frac{\E(f,f)}{\log(r^{-1})}.
\end{align*}
In particular, since $K\subset \W$ is bounded, for every $\varepsilon>0$ there is $0<r<c$ such that $\|f-\T_rf\|_{L^2}<\varepsilon$, for all $f\in K$. As each $\T_r(K)$ is totally bounded in $L^2(X,\mu)$ we conclude that $K$ is totally bounded as well.
\end{proof}

\begin{cor}
The operator $e^{-t\Ld}$ is compact for every $t>0$.
\end{cor}
We now describe the kernel and image of $\Delta$. Since $\mu(X)<\infty$, we have that $L^{\infty}(X,\mu)\subset L^{2}(X,\mu)$ and the projection $P:L^2(X,\mu)\to L^2(X,\mu)$ onto the constant functions is given by 
\begin{equation*}
Pf=\dashint_Xf(x)\dd\mu(x).
\end{equation*} 
Thus we have a decomposition $L^2(X,\mu)=\mathbb R\cdot 1\oplus \ker P$, where $1\in L^{\infty}(X,\mu)$ denotes the constant function and $\ker P$ is the subspace of square integrable functions with zero integral. 
\begin{prop}
\label{prop:kernel}
There is an orthogonal decomposition $L^2(X,\mu)=\ker \Ld \bigoplus \Im \Ld$, where $\ker \Ld=\Im P$ and $\Im \Ld = \ker P$. 
\end{prop}

\begin{proof}
We have that $L^2(X,\mu)=\ker \Ld \bigoplus \overline{\Im \Ld}$ since $\Ld$ is self-adjoint. By Proposition \ref{prop:compact_resolvent}, we obtain that $0$ is an isolated point of the spectrum of $\Delta$, so $\Im\Delta$ is in fact closed. For the kernel, note that for a self-adjoint operator $D$ we have $\ker D=\ker D^{2}$ and in particular $\ker \Delta^{1/2}=\ker\Delta$. The statement now follows since $f\in\ker \Delta^{1/2}$ if and only if $\mathcal{E}(f,f)=0$ if and only if $f$ is constant $\mu$-almost everywhere.
\end{proof}

\subsection{Eigenvalue asymptotics}
We now recall some facts about singular values of compact operators on separable Hilbert spaces from \cite[Chapter II]{GK}. The singular values of a compact operator $A:H_1\to H_2$ are the eigenvalues of $|A|=(A^*A)^{1/2}$. Let $(s_n(A))_{n\in \mathbb N}$ denote the sequence of its singular values in decreasing order, counting multiplicities. A useful fact is that for every $n\in \mathbb N$ it holds that
\begin{equation}\label{eq:singular_values_1}
s_n(A)=\inf_{F\in \R_{n-1}}\|A-F\|,
\end{equation}
where $\R_{n-1}$ is the set of operators from $H_1$ to $H_2$ of rank at most $n-1$ and $\|\cdot \|$ is the operator norm. Further, for every bounded operator $B:H_0\to H_1$ one has 
\begin{equation}\label{eq:singular_values_2}
s_n(AB)\leq s_n(A)\|B\|.
\end{equation}
Using the notion of singular values one can define several ideals of compact operators. In this paper, we are interested in the \textit{logarithmic integral ideal} which is the subset of the compact operators $\mathbb{K}(L^2(X,\mu))$ given by
\begin{equation*}\label{eq:singular_values_3}
\Li(L^2(X,\mu))=\{T\in \mathbb{K}(L^2(X,\mu)): s_n(T)= O((\log n)^{-1})\}
\end{equation*}
as well as the set
\begin{equation*}\label{eq:singular_values_4}
\Li^{\frac{1}{2}}(L^2(X,\mu))=\{T\in \mathbb{K}(L^2(X,\mu)): s_n(T)= O((\log n)^{-\frac{1}{2}})\}.
\end{equation*}
Both are two-sided ideals of the bounded operators on $L^2(X,\mu)$ and Banach spaces with norms described in \cite[p. 401]{Connes}. The latter ideal is called the \textit{square root} of the former as $T\in \Li^{\frac{1}{2}}(L^2(X,\mu))$ if and only if $|T|^2\in \Li(L^2(X,\mu))$. 

We will prove that $\Delta$ has resolvent in $\Li(L^2(X,\mu))$. For this we require the following geometric observation that holds more generally for doubling spaces.

\begin{lemma}[{\cite[Lemma 2.3]{Hy}}]
\label{lem: cover-upper-bound}
There exists $N\in\mathbb{N}$ such that for all $n\in\mathbb{N}$ we can cover $X$ by at most $N^{n}$ balls of radius $e^{-n}$.
\end{lemma}

We now prove the main theorem regarding the resolvent of $\Ld$.

\begin{thm}\label{theorem:resolvent}
The logarithmic Dirichlet Laplacian has resolvent in $\Li(L^2(X,\mu))$.
\end{thm}
\begin{proof}
Let us denote the embedding $\W\hookrightarrow L^2(X,\mu)$ by $J$. We first prove that 
\begin{equation}\label{eq:singular_values_5}
s_n(J)=O((\log n)^{-\frac{1}{2}})
\end{equation}
by approximating $J$ with finite rank operators. The proof consists of a refinement of the approximation argument in Proposition \ref{prop:compact_resolvent}. For $0<r<c$, recall the kernel $t_r:X\times X\to \mathbb R$, its normalisation $T_r:X\times X\to \mathbb R$ and the integral operator $\T_r:L^2(X,\mu)\to L^2(X,\mu)$ with kernel $T_r$.

By Lemma \ref{lem: cover-upper-bound} we can find a sequence $\mathcal{B}_n$ of covers of $X$ by open balls of radius $e^{-n}$ such that $|\mathcal{B}_n|\leq N^n$. Consider a subordinated continuous partition of unity $\{\chi_B\}_{B\in \mathcal{B}_{n}}$. Then, we can define the bounded operator $P_n:L^2(X,\mu)\to L^2(X,\mu)$ given by
\begin{equation*}
P_nf=\sum_{B\in \mathcal{B}_n} \left(\dashint_B f(z) \dd\mu (z)\right) \chi_B.
\end{equation*} 
Clearly, $P_n$ is finite rank of rank at most $|\mathcal{B}_{n}|\leq N^{n}$. Consider now the integral operator $Q_n=P_n\T_{e^{-n}}$ for large enough $n\in \mathbb N$ so that $e^{-n}<c$, which is also of rank at most $N^n$. Its kernel $q_n:X\times X\to \mathbb R$ is given by
\begin{equation*}
q_n(x,y)=\sum_{B\in \mathcal{B}_n}\left( \dashint_{B} T_{e^{-n}}(z,y) \dd\mu (z) \right) \chi_B(x).
\end{equation*}
Since for every $x\in X$ we have $\|T_{e^{-n}}(x,\cdot)\|_{L^1}=1$, we also obtain that $\|q_n(x,\cdot)\|_{L^1}=1$. Moreover, from the definition of the kernel $T_{e^{-n}}$ as the normalisation of $t_{e^{-n}}$ and the fact that $\|t_{e^{-n}}(x,\cdot)\|_{L^1}\simeq n$ we obtain that 
\begin{equation}\label{eq:singular_values_9}
q_n(x,y)\lesssim \frac{1}{n} \sum_{B\in \mathcal{B}_n}\left( \dashint_{B} t_{e^{-n}}(z,y) \dd\mu (z) \right) \chi_B(x).
\end{equation}

Now observe that for any $B\in \mathcal{B}_n$ that contains $x$ but not necessarily $y$ we have
\begin{equation}\label{eq:singular_values_10}
\dashint_{B} t_{e^{-n}}(z,y) \dd\mu (z)\lesssim d(x,y)^{-\delta}.
\end{equation}
This follows from the fact that $t_{e^{-n}}(z,y)=d(z,y)^{-\delta}$, if $d(z,y)\geq e^{-n}$, in which case $$\left(\frac{d(x,y)}{d(z,y)}\right)^{\delta}\leq \left(\frac{d(x,z)+d(z,y)}{d(z,y)}\right)^{\delta}<(2e^{-n}e^n+1)^{\delta}=3^{\delta},$$ since $x,z\in B$, while if $d(z,y)< e^{-n}$ then $t_{e^{-n}}(z,y)=0$. Inequalities \eqref{eq:singular_values_9} and \eqref{eq:singular_values_10} then yield that for $x,y\in X$ it holds $q_n(x,y)\lesssim d(x,y)^{-\delta} n^{-1}$. Now, working as in the proof of Proposition \ref{prop:compact_resolvent}, for every $f\in \W$ we obtain
\begin{align*}
\|f-Q_nf\|_{L^2}^2&\leq \int_X\int_X |f(x)-f(y)|^2 q_n(x,y)\dd\mu(y) \dd\mu(x)\\
&\lesssim \frac{\E(f,f)}{n}.
\end{align*}
This fact together with the finite rank description \eqref{eq:singular_values_1} of the singular values give that 
\begin{equation*}
s_{N^n}(J)\leq \|J-Q_nJ\| \lesssim n^{-\frac{1}{2}},
\end{equation*}
which leads to the desired equality \eqref{eq:singular_values_5} for the singular values of $J$. 

Now, since $(1+\Ld^{1/2})^{-1}:L^2(X,\mu)\to L^2(X,\mu)$ is the composition of a bounded operator from $L^2(X,\mu)\to \W$ with $J$, from inequality \eqref{eq:singular_values_2} we obtain that 
\begin{equation*}
(1+\Ld^{1/2})^{-1}\in \Li^{\frac{1}{2}}(L^2(X,\mu)).
\end{equation*}
Finally, observe that $(1+\Ld)^{-1}=(1+\Ld^{1/2})^{-2}(1+\Ld^{1/2})^{2}(1+\Ld)^{-1}$ where $(1+\Ld^{1/2})^{2}(1+\Ld)^{-1}$ is bounded. Since $(1+\Ld^{1/2})^{-1}$ is positive then $(1+\Ld^{1/2})^{-2}\in \Li(L^2(X,\mu))$. Therefore, we obtain that $(1+\Ld)^{-1}\in \Li(L^2(X,\mu))$, thus completing the proof.
\end{proof}

A direct computation then yields the following result.

\begin{cor}\label{cor:cutoff}
There is some $t_0>0$ such that $e^{-t\Ld}$ is trace class for every $t>t_0$.
\end{cor}

\section{Analysis of the domain}
\label{sec:domain}
Typically, it is quite challenging to describe the domain of the generator of a Dirichlet form such as $\Ld$. Nevertheless, in this section we manage to obtain interesting information about its structure.
\subsection{Commutators and form domains}
We prove some general facts about a positive self-adjoint operator $\A$ on a Hilbert space $H$, that is associated to a densely defined closed bilinear form $\Q$ as in Subsection \ref{sec:Dir}. First, using standard methods, we observe that $\Dom \A$ can be described in terms of the form $\Q$ as follows.

\begin{lemma}\label{lem:Domain_in_terms_of_E}
The subspace $\Dom \A \subset \Dom \Q$ consists of all $f\in \Dom \Q$ for which there is $C_f>0$ such that, for all $g\in \Dom \Q$ it holds $$|\langle f,g\rangle_{\Q}|\leq C_f \|g\|_{H}.$$ 
\end{lemma}

\begin{proof}
Let $f\in \Dom \A$ and $g\in \Dom \Q$. Then, we can choose $C_f=\|(1+\A)f\|_{H}$ since $$|\langle f,g \rangle_{\Q}|=|\langle (1+\A)f,g\rangle_{H}|\leq \|(1+\A)f\|_{H}\|g\|_{H}.$$ Conversely, consider $f\in \Dom \Q$ with $C_f>0$ such that, for all $g\in \Dom \Q$, it holds $$|\langle f,g\rangle_{\Q}|\leq C_f \|g\|_{H}.$$ Since $\Dom \Q$ is dense in $H$, the functional on $\Dom \Q$ given by $g\mapsto \langle f,g\rangle_{\Q}$ extends to a bounded functional $\phi$ on $H$. The Riesz Representation Theorem then gives $f'\in H$ such that $\phi(g)=\langle f',g \rangle_{H}$, for all $g\in H$. In particular, for every $g\in \Dom \Q$ we have $$\langle f,g \rangle_{\Q}= \langle f',g \rangle_{H} = \langle (1+\A)^{-1} f', g\rangle_{\Q},$$ as the image of the bounded operator $(1+\A)^{-1}:H \to H$ is $\Dom \A.$ We conclude that $f=(1+\A)^{-1} f'$ and thus $f\in \Dom \A$. 
\end{proof}
\begin{prop} 
\label{prop: domain-commutator-form}
Suppose that $T\in \mathbb{B}(H)$ is such that $T,T^{*}:\Dom \Q \to\Dom \Q$.  
\begin{enumerate}
\item If for every $f\in \Dom \A$ there is a constant $C_{f,T}>0$ such that for every $g\in \Dom \Q$, $$|\Q(Tf,g)-\Q(f,T^{*}g)|\leq C_{f,T}\|g\|_{H},$$ then $T:\Dom\A \to \Dom \A$. 
\item If in addition each $C_{f,T}=C_{T}\|f\|_{H}$, with $C_{T}$ independent of $f$, then the commutator $[\A,T]:\Dom \A \to H$ extends to a bounded operator on $H$.
\end{enumerate}
\end{prop}

\begin{proof}
For part (1), given $f\in \Dom \A$ the assumption implies that $$|\langle Tf,g\rangle_{\Q}|\leq (\|(1+\A)f\|_{H}\|T\|+C_{f,T})\|g\|_{H}.$$ Therefore, $Tf\in\Dom\A$ by Lemma \ref{lem:Domain_in_terms_of_E}.

For part (2), we know by (1) that the commutator $[\A,T]:\Dom\A\to H$ is well-defined. For $f\in\Dom\A$ and $g\in \Dom \Q$ we have
\begin{align*}
\langle [\A,T]f,g\rangle_{H} &=\langle \A Tf, g\rangle_{H}-\langle \A f,T^{*}g\rangle_{H}=\Q(Tf,g)-\Q(f,T^{*}g),
\end{align*}
so that 
\begin{align*}
|\langle [\A,T]f,g\rangle_{H}|=|\Q(Tf,g)-\Q(f,T^{*}g)|\leq C_{T}\|f\|_{H}\|g\|_{H}.
\end{align*}
Thus $[\A,T]$ extends to a bounded operator.
\end{proof}

\subsection{The Banach algebra of Dini continuous functions}
The logarithmic Dirichlet Laplacian $\Delta$ and the so-called Dini continuous functions on $(X,d)$ are closely related. Before making this precise, we remind the reader of some basic properties of Dini functions and prove that they form a Banach algebra. The latter might be a known fact, but we were not able to find a reference in the literature. Denote by $\Cucb(X)$ the space of uniformly continuous bounded real-valued functions on $X$ equipped with the supremum norm $\|\cdot \|_{\infty}$.
 
\begin{definition} For $f\in \Cucb(X)$, the \emph{modulus of continuity} $\omega_f:[0,1]\to [0,\infty)$ is given by $\omega_f(t)=\sup\{|f(x)-f(y)|: d(x,y)\leq t\diam(X)\}$. Then, we say that $f$ is \textit{Dini continuous} if its modulus of continuity satisfies 
\begin{equation*}
\D(f):=\int_{0}^{1}\frac{\omega_f(t)}{t}\dd t<\infty.
\end{equation*}
We call $\D(f)$ the \emph{Dini constant} of $f$ and denote by $\D(X,d)$ the set of all Dini continuous functions. The \emph{Dini norm} of $f\in\D(X,d)$ is $\|f\|_{\D}:=\|f\|_{\infty} + \D(f).$
\end{definition}

We recall some well-known facts about Dini continuous functions.

\begin{lemma} 
\label{lem: Dini-properties}
A function $f\in \Cucb(X)$ is Dini continuous if and only if for all $\theta\in (0,1)$ the sequence,
$$\omega^{\theta}_{f}:\mathbb{N}\cup\{0\}\to \mathbb{R}, \quad n\mapsto \omega_{f}(\theta^{n}),$$
is in $\ell^{1}(\mathbb{N}\cup\{0\})$. In fact, $\D(f) \lesssim\|\omega^{\theta}_{f}\|_{\ell^1}\lesssim \|f\|_{\D}$ and for $p>1$ we have $\|\omega^{\theta}_{f}\|_{\ell^p}\lesssim \|f\|_{\D}$.
\end{lemma}

\begin{proof}
We write the unit interval as the union of the intervals $[\theta^{n+1},\theta^n]$, where $n\geq 0$. The inequality $\D(f)\lesssim \|\omega^{\theta}_{f}\|_{\ell^1}$ is straightforward. For the other direction, observe that $\omega_{f}(1)\leq 2\|f\|_{\infty}$, and
\begin{align*}
\D(f)=\int_{0}^{1}\frac{\omega_f(t)}{t}\dd t=\sum_{n=0}^{\infty}\int_{\theta^{n+1}}^{\theta^{n}}\frac{\omega_f(t)}{t}\dd t\geq (1-\theta)\sum_{n=1}^{\infty}\omega_{f}(\theta^{n}),
\end{align*}
so that 
\begin{align*}
2\|f\|_{\D}\geq \omega_{f}(1)+2(1-\theta)\sum_{n=1}^{\infty}\omega_{f}(\theta^{n})\geq \|\omega^{\theta}_{f}\|_{\ell^1}.
\end{align*}
The $\ell^{p}$ inequalities now follow from the fact that $\ell^{1}$ embeds into $\ell^{p}$ contractively,
thus completing the proof.
\end{proof}

\begin{prop}
For $f,g\in \Cucb(X)$ the modulus of continuity satisfies
\[\omega_{fg}(t)\leq \omega_{g}(t) \|f\|_{\infty}+ \omega_{f}(t)\|g\|_{\infty}.\]
Moreover, the vector space $\D(X,d)$ is a Banach algebra with respect to the norm 
\begin{equation*}
\|f\|_{\D}=\|f\|_{\infty} + \D(f).
\end{equation*}
\end{prop}

\begin{proof}
The first assertion follows since
\begin{align*}
|(fg)(x)-(fg)(y)|&\leq |g(x)-g(y)||f(x)|+|f(x)-f(y)||g(y)|,
\end{align*}
so that $\|fg\|_{\D}\leq \|f\|_{\D}\|g\|_{\D}$. 

To prove that $\D(X,d)$ is a Banach space let $(f_{n})_{n\in \mathbb N}$ be a Cauchy sequence in $\D(X,d)$. Then, $\sup_{n} \D(f_n)<\infty$ and $(f_{n})_{n\in \mathbb N}$ is Cauchy with respect to $\|\cdot \|_{\infty}$ and hence there is $f\in \Cucb(X)$ such that $\|f-f_n\|_{\infty}\to 0$. In particular, for each $t>0$ we can choose $n(t)$ large enough so that $\|f-f_{n(t)}\|_{\infty}<t/2$. Then, we have $\omega_{f}(t)\leq \omega_{f_{n(t)}}(t)+t$ so that for all $\varepsilon>0$ we have
\[\int_{\varepsilon}^{1}\frac{\omega_{f}(t)}{t}\dd t\leq 1+\D(f_{n(t)})\leq 1+\sup_{n} \D(f_n),\]
and therefore $\D(f)<\infty$, so $f\in \D(X,d)$.

Lastly, we need to show that $\D(f-f_n)\to 0$. To this end, note from the assumption that $\D(f_m-f_n)\to 0$ when $n,m\to \infty$. Since $\|f-f_{n}\|_{\infty}\to 0$, a simple argument shows that $\omega_{f_n}(t)\to \omega_{f}(t)$. Using Fatou's Lemma we write
\begin{align*}
\D(f-f_n)&=\int_{0}^{1}\lim_m\frac{\omega_{f_m-f_n}(t)}{t}\dd t\leq \liminf_m \int_{0}^{1}\frac{\omega_{f_m-f_n}(t)}{t}\dd t= \liminf_m \D(f_m-f_n),
\end{align*}
where the latter tends to zero as $n$ goes to infinity.
\end{proof}

\subsection{Module structure over $\D(X,d)$} 

An interesting feature of $\Dom\Ld$ is that it is closed under left multiplication by elements in $\D(X,d)$. In fact, we show that $\Dom\Ld$ is a Banach module over $\D(X,d)$. For $h\in L^{\infty}(X,\mu)$ consider the multiplication operator $\m_h:L^2(X,\mu)\to L^2(X,\mu)$ given by $\m_hf=hf$, so that $\m_{h}^{*}=\m_{h}$.

First, we prove that $\D(X,d)$ preserves the domain of the Dirichlet form $\E$.

\begin{lemma}\label{lem:form_domain_module}
Let $h\in \D(X,d)$ and $f\in \W$. Then, $\m_h f=hf\in \W$ and $$\|hf\|_{\E}\lesssim \|h\|_{\D}\|f\|_{\E}.$$
\end{lemma}

\begin{proof}
Since $h$ is bounded we have that $hf\in L^2(X,\mu)$. Now,
\begin{align*}
\E(hf,hf)&=\frac{1}{2}\int_X\int_X\frac{|h(x)f(x)-h(y)f(y)|^2}{d(x,y)^{\delta}} \dd\mu(y)\dd\mu(x)\\
&\leq \int_X\int_X\frac{|f(x)(h(x)-h(y))|^2+|h(y)(f(x)-f(y))|^2}{d(x,y)^{\delta}}\dd\mu(y)\dd\mu(x)\\
&\lesssim  \int_X\int_X\frac{|f(x)|^{2}|(h(x)-h(y))|^2}{d(x,y)^{\delta}}\dd\mu(y)\dd\mu(x)+\|h\|_{\infty}^2\E(f,f),
\end{align*}
and we now prove that the remaining integral is finite.  For every $x\in X$ and integer $k\geq 0$, define $r_k=e^{-k}\diam(X)$ and the annulus $B_{x,k}=B(x,r_k)\setminus B(x,r_{k+1})$. Using Lemmas \ref{lem:Ahlfors_estimates} and \ref{lem: Dini-properties} we then find
\begin{align*}
\int_X\int_X\frac{|f(x)|^{2}|h(x)-h(y)|^2}{d(x,y)^{\delta}}\dd\mu(y)\dd\mu(x)
&= \int_X\sum_{k=0}^{\infty}\int_{B_{x,k}}\frac{|f(x)|^{2}|h(x)-h(y)|^2}{d(x,y)^{\delta}}\dd\mu(y)\dd\mu(x)\\
&\leq \int_X\sum_{k=0}^{\infty}\omega_h(e^{-k})^2\int_{B_{x,k}}\frac{|f(x)|^{2}}{d(x,y)^{\delta}}\dd\mu(y)\dd\mu(x)\\
&\lesssim \|f\|^{2}_{L^{2}}\sum_{k=0}^{\infty}\omega_h(e^{-k})^2\\
&\lesssim \|h\|_{\D}^{2}\|f\|^{2}_{L^{2}}.
\end{align*}
Thus we have shown that
\[\mathcal{E}(hf,hf)\lesssim \|h\|_{\D}^{2}\|f\|^{2}_{L^{2}}+\|h\|_{\infty}^2\E(f,f), \]
which implies $\|hf\|_{\mathcal{E}}\lesssim \|h\|_{\D}\|f\|_{\mathcal{E}}$.
\end{proof}
\begin{cor}\label{cor:Dini_embed_1}
There is a continuous inclusion $\D(X,d)\hookrightarrow\W$.
\end{cor}
\begin{proof} Since the constant function $1\in\W$ and $\mathcal{E}(1,1)=0$, for $h\in\D(X,d)$ we have $h=h\cdot 1\in\W$ and $\| h\|_{\mathcal{E}}\lesssim \|h\|_{\D}$.
\end{proof}
\begin{lemma}\label{lem:commutator_bounded}
Let $h\in \D(X,d)$. The kernel
\[K_{h}:X\times X\setminus D\to \mathbb{R},\,\,K_{h}(x,y):=\frac{h(x)-h(y)}{d(x,y)^{\delta}},\]
defines a bounded operator $\K_{h}:L^2(X,\mu)\to L^2(X,\mu)$ given by $$\K_{h}f(x)=\int_X\frac{h(x)-h(y)}{d(x,y)^{\delta}}f(y)\dd\mu(y),$$ with operator norm $\|\K_{h}\|\lesssim \|h\|_{\D}.$ 
\end{lemma}

\begin{proof}
First we show that the operator $\K_h:L^2(X,\mu)\to L^2(X,\mu)$ is well-defined and bounded. Indeed, let $x\in X$ and for every integer $k\geq 0$ define $r_k=e^{-k}\diam(X)$ as well as the annulus $B_{x,k}=B(x,r_k)\setminus B(x,r_{k+1}).$ Then, we have that 
\begin{align*}
|\K_h f(x)|&\leq \sum_{k=0}^{\infty}\int_{B_{x,k}}\frac{|h(x)-h(y)|}{d(x,y)^{\delta}}|f(y)|\dd\mu(y)\\
&\lesssim \sum_{k=0}^{\infty} \omega_h(e^{-k})\dashint_{B(x,r_k)}|f(y)|\dd\mu(y)\\
&\lesssim \|h\|_{\D}\M(f)(x),
\end{align*}
where $\M(f)$ is the maximal function of $f\in L^2(X,\mu)$ given by 
\begin{equation*}
\M(f)(x):=\sup_{r>0}\dashint_{B(x,r)}|f(y)|\dd\mu(y).
\end{equation*}
From Hardy-Littlewood's Maximal Function Theorem \cite[Theorems 2.2]{Hei} we obtain that $\|\M(f)\|_{L^2}\lesssim \|f\|_{L^2}.$ Consequently, we have 
$$\|\K_{h}f\|_{L^2}\lesssim \|h\|_{\D} \|f\|_{L^2},$$
which completes the proof.
\end{proof}

By Lemma \ref{lem:form_domain_module}, $\m_h:\W\to \W$ whenever $h\in\D(X,d)$. In fact, the following holds.

\begin{thm}\label{theorem:domain_module}
Let $h\in \D(X,d)$. Then, $\m_{h}:\Dom \Ld\to \Dom\Ld$ and the commutator $[\Ld,\m_{h}]:\Dom\Ld\to L^{2}(X,\mu)$ extends to the bounded operator $\K_h$.
\end{thm}

\begin{proof}
Let $f\in \Dom \Delta$. By Lemma \ref{lem:form_domain_module} we have that $hf\in \W$, so by Proposition \ref{prop: domain-commutator-form} it suffices to show that there is $C_{h}>0$ such that, for all $g\in \W$, we have 
\begin{equation*}
|\E(hf,g)-\mathcal{E}(f,hg)|\leq C_{h}\| f\|_{L^{2}}\|g\|_{L^2}.
\end{equation*}
Subtracting the integrands shows that
\begin{align}
\nonumber\E(hf,g)-\mathcal{E}(f,hg)&=\frac{1}{2}\int_{X}\int_{X}\frac{(h(y)-h(x))f(x)g(y)+(h(x)-h(y))f(y)g(x)}{d(x,y)^{\delta}}\mathrm{d}\mu(y)\mathrm{d}\mu(x)\\ 
\label{form-integral} &=\int_{X}\int_{X}\frac{(h(x)-h(y))f(y)g(x)}{d(x,y)^{\delta}}\mathrm{d}\mu(y)\mathrm{d}\mu(x)
\end{align}
where the second equality follows since \eqref{form-integral} is finite by Lemma \ref{lem:commutator_bounded}, and the change of variables $(x,y)\mapsto (y,x)$.
Therefore, by Lemma \ref{lem:commutator_bounded} we find
\begin{align*}
|\nonumber\E(hf,g)-\mathcal{E}(f,hg)|=|\langle \K_{h}f,g\rangle_{L^{2}}|\lesssim \|h\|_{\D}\|f\|_{L^{2}}\|g\|_{L^{2}},
\end{align*}
as desired.
\end{proof}
\begin{cor}\label{cor:Din_in_Dom} There is a continuous inclusion $\D(X,d)\hookrightarrow\Dom\Delta$.
\end{cor}
\begin{proof}
For $h\in\D(X,d)$ write $h=\m_h\cdot 1\in\Dom\Delta$, as $1\in\Dom\Delta$. Then, since $1\in \ker \Delta$,
$$\|h\|_{L^{2}}+\|\Delta h\|_{L^{2}}\lesssim\|h\|_{\infty}+\|[\Delta,\m_h]\cdot 1\|_{L^{2}}\lesssim \|h\|_{\infty}+\|[\Delta,\m_{h}]\|\lesssim\|h\|_{\D},$$
as claimed.
\end{proof}

\subsection{Integral representation}\label{sec:integralrep}
We have seen that $\Dom\Ld$ contains the space of Dini continuous functions on $(X,d)$. In this section we show that $\Ld$ admits a singular integral representation on those functions. We emphasise that this integral representation does not determine $\Ld$, unless $\D(X,d)\subset \Dom \Ld$ is a core for $\Ld$. The latter we only know to be true in special cases, see Section \ref{sec:Examples}.
\begin{prop}\label{prop:integral_rep}
For every $f\in \D(X,d)\subset\Dom\Ld$ and $x\in X$ we have the principal value integral representation $$\Ld f(x)=\int_X \frac{f(x)-f(y)}{d(x,y)^{\delta}}\dd\mu(y),$$
and $\Ld f\in L^{\infty}(X,\mu)\subset L^{2}(X,\mu)$.
\end{prop}

\begin{proof}
Let $\Ld_0:\D(X,d)\subset L^2(X,\mu)\to L^2(X,\mu)$ be the densely defined operator given by $$\Ld_0f(x)=\int_X \frac{f(x)-f(y)}{d(x,y)^{\delta}}\dd\mu(y).$$ To see that $\Delta_0(f)\in L^{2}(X,\mu)$ we observe that $\Delta_0(f)=\K_f(1)$, where $\K_f$ is the bounded operator of Lemma \ref{lem:commutator_bounded}, hence $|\Ld_0f(x)|\lesssim \|f\|_{\D},$ so that $\Delta_0(f)\in L^{\infty}(X,\mu)$. Then, for $f\in \D(X,d)$ and $g\in \W$ one observes that $\E(f,g)=\langle \Ld_0f, g\rangle_{L^2}$. In particular, from Corollary \ref{cor:Din_in_Dom} we obtain that for every $g\in \Dom \E$ it holds $$\langle \Delta f,  g\rangle_{L^2} = \langle \Ld_0f, g\rangle_{L^2}.$$ Consequently, $\Delta f=\Delta_0 f$.
\end{proof}

\subsection{Smooth vectors}\label{sec:smooth}
Recall from \cite[Section X.6]{RS} that the set of \textit{smooth vectors for} $\Ld$ is given by $$C^{\infty}(\Ld)=\bigcap_{n=1}^{\infty}\Dom(\Ld^n).$$ Since $\Ld$ is self-adjoint, the set $C^{\infty}(\Ld)$ is dense in $\Dom(\Ld)$. Here we show that $C^{\infty}(\Ld)$ contains the set of H\"older continuous functions on $X$.

First, let $0<\alpha \leq 1$ and $f\in \Cucb(X)$. We say that $f$ is \textit{almost} $\alpha$\textit{-H\"older} if it is $\beta$-H\"older for every $0<\beta <\alpha$. The set of almost $\alpha$-H\"older functions will be denoted by $\Hol_{<\alpha}(X,d)$. We remark that since $\diam(X)<\infty$ one has $\Hol_{\alpha}(X,d)\subset \Hol_{<\alpha}(X,d)$.

\begin{prop}\label{prop:smooth_vectors}
For every $0<\alpha\leq 1$ we have $$\Ld(\Hol_{<\alpha}(X,d))\subset \Hol_{<\alpha}(X,d)$$ and hence $\Hol_{<\alpha}(X,d)\subset C^{\infty}(\Ld).$ 
\end{prop}

\begin{proof}
We assume without loss of generality that $\diam(X)<3^{-1}$. Let $f\in \Hol_{<\alpha}(X,d)$ and $0< \varepsilon <\alpha$. We claim that there is $\Lambda_{f,\varepsilon}>0$ such that 
\begin{equation}\label{eq:smooth_vectors_1}
|\Ld f(x)-\Ld f(y)|\leq \Lambda_{f,\varepsilon}d(x,y)^{\varepsilon}.
\end{equation}
To this end, let $0<\beta <\alpha$ be an arbitrary exponent such that $$|f(x)-f(y)|\leq C_{f,\beta}d(x,y)^{\beta},$$ for some $C_{f,\beta}>0$. Also, fix $x\neq y \in X$ and define $r=d(x,y)$ as well as $B=B(y,3r).$ Then, using a decomposition similar to the one used in \cite[Theorem 1.2]{GSV} we write 
\begin{align*}
\Ld f(x) - \Ld f(y)&=\int_B\frac{f(x)-f(z)}{d(x,z)^{\delta}}\dd \mu(z) - \int_B\frac{f(y)-f(z)}{d(y,z)^{\delta}}\dd\mu(z)\\
&\quad+ \int_{X\setminus B} \frac{f(x)-f(z)}{d(x,z)^{\delta}}-\frac{f(y)-f(z)}{d(y,z)^{\delta}} \dd\mu (z),
\end{align*}
where we denote the integrals by $\I_1,\I_2$ and $\I_3$, respectively. Also, let us denote by $C_s$ the constant $Ce^{\delta +s}(e^s-1)^{-1}$ found in parts (1) and (2) of Lemma \ref{lem:Ahlfors_estimates}.

It is straightforward to check that 
\begin{align*}
\tag{i} |\I_1|&\lesssim C_{f,\beta} C_{\beta}d(x,y)^{\beta}\\
\tag{ii} |\I_2|&\lesssim C_{f,\beta} C_{\beta}d(x,y)^{\beta}.
\end{align*}
Further, we write $\I_3$ as the sum of the integrals $\J_1$ and $\J_2$, where 
\begin{align*}
\J_1&=\int_{X\setminus B}\frac{f(x)-f(y)}{d(x,z)^{\delta}}\dd \mu(z)\\
\J_2&=\int_{X\setminus B}\left( f(y)-f(z)\right) \left(\frac{1}{d(x,z)^{\delta}}-\frac{1}{d(y,z)^{\delta}}\right) \dd \mu (z).
\end{align*}
Before estimating $|\J_1|$ and $|\J_2|$, we note that for every $0<t\leq 1$ and $0<\gamma <1$ one has
\begin{equation}\label{eq:smooth_vectors_2}
t\log(t^{-1})\leq \frac{1}{e(1-\gamma)}t^{\gamma}.
\end{equation}
Indeed, the function $\psi:(0,1]\to \mathbb R$ given by $\psi(t)=t^{1-\gamma}\log(t^{-1})$ has a maximum at $t=e^{-(1-\gamma)^{-1}}$ and $$\psi(e^{-(1-\gamma)^{-1}})=\frac{1}{e(1-\gamma)}.$$

From Lemma \ref{lem:Ahlfors_estimates} and inequality (\ref{eq:smooth_vectors_2}), for every $0<\gamma <1$ we have that
\begin{align*}
\tag{iii} |\J_1|&\lesssim C_{f,\beta}d(x,y)^{\beta}\log((3d(x,y))^{-1})\\
&\leq \frac{C_{f,\beta}}{\beta}d(x,y)^{\beta}\log(d(x,y)^{-\beta})\\
&\lesssim \frac{C_{f,\beta}}{\beta(1-\gamma)}d(x,y)^{\beta \gamma}.
\end{align*}
Now, observe as in \cite[Lemma 2.3]{GSV} that for $z\in X\setminus B$ we have that 
\begin{equation}\label{eq:smooth_vectors_3}
\left| \frac{1}{d(x,z)^{\delta}}-\frac{1}{d(y,z)^{\delta}}\right|\lesssim d(x,y)d(y,z)^{-\delta-1}.
\end{equation}
Consequently, from Lemma \ref{lem:Ahlfors_estimates} and inequality (\ref{eq:smooth_vectors_3}) we obtain 
\begin{align*}
\tag{iv} |\J_2|&\lesssim C_{f,\beta} d(x,y) \int_{X\setminus B} \frac{1}{d(y,z)^{\delta +1-\beta}}\dd \mu(z)\\
&\lesssim C_{f,\beta}C_{1-\beta}d(x,y)(3d(x,y))^{-1+\beta}\\
&\lesssim C_{f,\beta}C_{1-\beta}d(x,y)^{\beta}.
\end{align*}

Since $C_{\beta}, C_{1-\beta}$ can be bounded above (as functions of $\beta$) by some multiple of $\beta^{-1}$, from (i), (ii), (iii) and (iv) it follows that 
\begin{equation*}
\lvert \Ld f(x) - \Ld f(y) \rvert \lesssim \frac{C_{f,\beta}}{\beta(1-\gamma)}d(x,y)^{\beta \gamma}.
\end{equation*}
By choosing $\beta=(\alpha+\varepsilon)/2$ and $\gamma=2\varepsilon/(\alpha+\varepsilon)$, we obtain \eqref{eq:smooth_vectors_1} and the proof is complete.
\end{proof}

\section{Examples}\label{sec:Examples}

We present several examples of Ahlfors regular metric-measure spaces on which $\Ld$ can be well understood and has particularly interesting properties. For these examples, the Lipschitz continuous functions turn out to be a core of $\Ld$, which is equivalent to saying that the restriction $\Ld_0: \Lip(X,d)\to L^2(X,\mu)$ given by 
\begin{equation}\label{eq:small_Delta}
\Ld_0 f(x)=\int_X \frac{f(x)-f(y)}{d(x,y)^{\delta}} \dd\mu(y)
\end{equation}
is essentially self-adjoint and $\Ld$ is its unique self-adjoint extension. Moreover, in most examples we obtain sharp estimates for the threshold $t_0>0$ of Corollary \ref{cor:cutoff}, meaning that the compact operators $e^{-t\Ld}$ are trace class if and only if $t>t_0$. 
\subsection{Ahlfors spaces as noncommutative manifolds}\label{sec:ncmanifold}
We provide an interpretation of the results of this paper in the language of Alain Connes' noncommutative geometry \cite{Connes}. This interpretation has been the guiding idea for our work, but our results and the remaining examples below are independent of it.
\begin{definition} A \emph{spectral triple} $(\mathcal{A}, H, D)$ consists of a complex unital $*$-algebra $\mathcal{A}$, a complex Hilbert space $H$ and a self-adjoint operator $D:\Dom D\to H$ such that
\begin{enumerate}
\item $\mathcal{A}\subset \mathbb{B}(H)$;
\item the operator $D$ has compact resolvent; 
\item for all $a\in\mathcal{A}$, $a:\Dom D\to \Dom D$ and $[D,a]$ extends to a bounded operator.
\end{enumerate}
\end{definition}
This definition is motivated by first order elliptic operators on compact Riemannian manifolds. If $M$ is such a manifold, we can choose $\mathcal{A}=C^{1}(M)$, $H=L^{2}(M,\bigwedge^{*}T^{*}M)$ the Hilbert space of $L^{2}$-differential forms on $M$, and $D=d+d^{*}$ the Hodge DeRham operator. Then $(C^{1}(M),L^{2}(M,\bigwedge^{*}T^{*}M),d+d^{*})$ is a spectral triple from which the Riemannian distance, the dimension and the Euler characteristic of $M$ can be recovered \cite{Connes}. This leads to the viewpoint that for a general $*$-algebra $\mathcal{A}$, a spectral triple endows $\mathcal{A}$ with the structure of a \emph{noncommutative manifold}.
\begin{definition}
\label{def: summable} 
A self-adjoint operator $D$ is 
\begin{enumerate}
\item \emph{finitely summable} if for some $t>0$ the operator $(1+|D|)^{-t}$ is trace class;
\item \emph{$\Li$-summable} if for some $t>0$ the operator $e^{-t|D|}$  is trace class.
\end{enumerate}
A spectral triple $(\mathcal{A},H,D)$ is \textit{finitely} or $\Li$\textit{-summable} if its operator $D$ is.
\end{definition}
In Definition \ref{def: summable} we use the absolute value $|D|=(D^{*}D)^{1/2}$, because the operator in a spectral triple need not be positive. Observe that finite summability implies $\Li$-summability, but not the other way around.

Using the embedding $\D_{\mathbb{C}}(X,d)\hookrightarrow \mathbb{B}(L^{2}_{\mathbb{C}}(X,\mu))$, Proposition \ref{prop:compact_resolvent} and Theorems \ref{theorem:resolvent} and \ref{theorem:domain_module} now combine to the statement that Ahlfors spaces are noncommutative manifolds.
\begin{prop}\label{prop:spectraltriple}
Let $(X,d,\mu)$ be an Ahlfors regular metric-measure space and $\Delta_{\mathbb{C}}$ be its complex logarithmic Dirichlet Laplacian. Then, the triple $(\D_{\mathbb{C}}(X,d), L^{2}_{\mathbb{C}}(X,\mu), \Delta_{\mathbb{C}})$ is an $\Li$-summable spectral triple.
\end{prop}
Denote by $\mathrm{L}^{1}(H)$ the ideal of trace class operators. The thresholds
\begin{equation}
\label{eq: thresholds}
\inf\{t>0:(1+|D|)^{-t}\in \mathrm{L}^{1}(H)\},\quad \inf\{t>0: e^{-t|D|}\in\mathrm{L}^{1}(H)\},
\end{equation}
encode a notion of spectral dimension. For a compact Riemannian manifold $(M,g)$ it is well-known that the Laplace-Beltrami operator $\Delta_{g}$ satisfies 
\begin{equation}\label{eq: dimension}
\mathrm{dim}(M)=2\inf\{t>0:(1+\Delta_{g})^{-t}\in \mathrm{L}^{1}(H)\}.
\end{equation}
We now illustrate the relation between dimension and the threshold \eqref{eq: thresholds} in the $\Li$-summable case. Recall that for a compact Riemannian manifold $(M,g)$, the Riemannian distance $\rho_{g}$ and the volume measure $\mu_{g}$ induce an Ahlfors regular metric-measure space $(M, \rho_{g}, \mu_{g})$. As such, $M$ carries a logarithmic Dirichlet Laplacian $\Delta$. 
\begin{prop}[\cite{GU}]
\label{prop: manifold} 
Let $(M,g)$ be a compact $n$-dimensional Riemannian manifold. Then $\Delta$ is a bounded perturbation of $\frac{\pi^{n/2}}{\Gamma(n/2)}\log (1+\Delta_{g})$. Therefore, $\Delta$ is essentially self-adjoint on $C^{\infty}(M)$ and $e^{-t\Delta}$ is trace class if and only if $t>\frac{n\Gamma(n/2)}{2\pi^{n/2}}$.
\end{prop}
\begin{proof} By \cite[Example 2.9]{GU}, $\Delta$ is a bounded perturbation of $c\log (1+\Delta_{g})$ for some $c>0$. The constant $c=\frac{\pi^{n/2}}{\Gamma(n/2)},$ which can be derived using the heat kernel methods of \cite{ACM} or using Fourier theory as in \cite[Lemma 25.2]{SKM}. Since $\Delta$ and $c\log (1+\Delta_{g})$ are positive with compact resolvent and $B:=\Delta-c\log(1+\Delta_{g})$ is bounded, the Min-Max principle shows that the $n$-th eigenvalues $\lambda_n,\lambda_n'$ of $\Delta$ and $c\log (1+\Delta_{g})$ respectively, counting multiplicities, satisfy $|\lambda_n-\lambda_n'|\leq \|B\|$. Then, using equation \eqref{eq: dimension}
we find the threshold
\begin{equation}\label{eq:thres}
\inf\{t>0: e^{-t\Delta}\in \mathrm{L}^{1}\}=\frac{n\Gamma(n/2)}{2\pi^{n/2}},
\end{equation} 
as claimed.
\end{proof}

\begin{remark}
In fact, for $n\geq 2$ one has $\frac{2\pi^{n/2}}{\Gamma(n/2)}$ equals the Riemannian volume $\text{vol}(\mathbb S^{n-1})$, where $\mathbb S^{n-1}$ is the unit $(n-1)$-sphere in $\mathbb R^{n}$, and thus the threshold in \eqref{eq:thres} is equal to $\frac{n}{\text{vol}(\mathbb S^{n-1})}.$
\end{remark}

\subsection{Full shift spaces}\label{sec:shift}

Let $N\in \mathbb N$ and equip $\{1,\ldots, N\}$ with the discrete topology and the space of infinite words $X=\{1,\ldots, N\}^{\mathbb N}$ with the product topology, with respect to which it is a totally disconnected compact Hausdorff space. For a finite word $w=w_1\ldots w_n$ where each $w_k\in \{1,\ldots, N\}$, denote its length $n$ by $|w|$ and consider the cylinder set 
\begin{equation*}
C_{w}=\{x\in X: x_k=w_k,\,\, \text{for}\,\, 1\leq k\leq n\}.
\end{equation*} 
The cylinder sets are clopen in $X$ and form a basis for the topology on $X$. The space $X$ is known as the \textit{full} $N$\textit{-shift} as it is naturally equipped with the left shift map $\sigma:X\to X$ given by $\sigma(x)_k=x_{k+1}.$ It is a local homeomorphism with topological entropy $\h(\sigma)=\log N$, see \cite{KH}. Further, for every $\lambda >1$, the topology on $X$ is induced by the ultrametric
\begin{equation*}
d(x,y)= \lambda^{-\inf \{k-1:x_k\neq y_k\}}.
\end{equation*}
In fact, for every $x\in X$ and $n\in \mathbb N$ the clopen ball $B(x,\lambda^{-n})=C_{x_1\ldots x_n}.$ Finally, consider the Bernoulli measure $\mu$ on $(X,d)$ given on cylinder sets by $\mu (C_w)=N^{-|w|}.$ It is evident that $\mu$ is Ahlfors $\delta$-regular with 
\begin{equation*}
\delta=\frac{\log N}{\log \lambda}.
\end{equation*}
We now diagonalise $\Ld$ by constructing a Haar basis for $L^2(X,\mu)$ using cylinder sets and viewing $\Ld$ as a dyadic operator like in \cite{AA}. We should note though that in general, the operator $\Ld$ is not dyadic, since $(X,d)$ is not necessarily zero-dimensional.

At this point we make the convention that $X$ is the cylinder set of the empty word with length zero. For $n\geq 0$, let $V_n$ be the subspace of $L^2(X,\mu)$ spanned by the characteristic functions of the sets $C_w$ with $|w|=n$. Notice that $V_{n}\subset V_{n+1}$ for all $n\geq 0$ and that their union is dense in $L^2(X,\mu)$. Denote now by $P_{n}$ the projection of $L^2(X,\mu)$ onto $V_{n}$ and consider the projection $Q_{n}=P_{n+1}-P_{n}$ which has rank $N^{n}(N-1)$, for $n\geq 0$. Also, observe that 
\begin{equation*}
\bigoplus_{n\geq 0} \Im Q_{n}=\ker P_0
\end{equation*}
and recall from Proposition \ref{prop:kernel} that $\ker P_0=\Im \Ld$. Working as in \cite{AA} we obtain the following.

\begin{lemma}\label{lem:wavelets}
There is an orthonormal basis $H_{n}$ of $\Im Q_{n}$, consisting of Lipschitz functions $h\in L^2(X,\mu)$ such that
\begin{enumerate}[(i)]
\item the support of $h$ lies in a unique cylinder set $C_w(h)$ with $|w|=n$;
\item $h$ is constant on each cylinder set $C_{w'}\subset C_w(h)$ with $|w'|=n+1$;
\item $\int_X h(x)\dd\mu (x)=0$.
\end{enumerate}
\end{lemma}

The Lipschitz functions in Lemma \ref{lem:wavelets} are usually called \textit{Haar wavelets} and we claim that are eigenfunctions of $\Delta$.

\begin{prop}\label{prop:shift_diagonalisation}
The projections $Q_{n}$ yield a spectral decomposition of $\Ld$. Consequently, the Lipschitz functions on $(X,d)$ are a core for $\Ld$ and the compact operator $e^{-t\Ld}$ is trace class if and only if 
$$t>\frac{N\log N}{N-1}.$$
\end{prop}

\begin{proof}
Let $h\in H_n$ and $x\in X$. Working as in the proof of \cite[Theorem 3.1]{AA}, we obtain that if $x\in X\setminus C_w(h)$ then $h(x)=0$ and $\Ld h(x)=0$. Further, if $x\in C_w(h)$ then 
\begin{equation}\label{eq:shift_1}
\Ld h(x)=\left(1+\int_{X\setminus C_w(h)}\frac{1}{d(x,y)^{\delta}}\dd\mu (y)\right) h(x).
\end{equation}
Now notice that the coefficient in \eqref{eq:shift_1} is constant and in our case it can be computed explicitly. Indeed, for $n=0$ we have $C_w(h)=X$ and hence $\Ld h=h$. For $n\in \mathbb N$ observe that $C_w(h)=B(x,\lambda^{-n})$ and denoting $B_{x,k}=B(x,\lambda^{-n+k+1})\setminus B(x,\lambda^{-n+k})$, for $0\leq k\leq n-1$, we can write
\begin{align*}
\int_{X\setminus C_w(h)}\frac{1}{d(x,y)^{\delta}}\dd\mu (y)&= \sum_{k=0}^{n-1}\int_{B_{x,k}}\frac{1}{d(x,y)^{\delta}}\dd\mu (y)\\
&=\sum_{k=0}^{n-1}\lambda^{\delta(n-k-1)}\left(\mu(B(x,\lambda^{-n+k+1}))-\mu(B(x,\lambda^{-n+k}))\right)\\
&=\sum_{k=0}^{n-1}N^{n-k-1}(N^{-n+k+1}-N^{-n+k})\\
&=\left(1-\frac{1}{N}\right)n.
\end{align*}
As a result, for every $f\in \bigcup_{n\geq 0} V_n$ we obtain that
\begin{equation}\label{eq:shift_2}
\Ld f=\sum_{n\geq 0} \left(1+\left(1-\frac{1}{N}\right)n\right)Q_n f.
\end{equation}
The fact that the Lipschitz functions are a core for $\Ld$ then follows. The calculation of the threshold $t_0=(N-1)^{-1}N\log N$ follows from knowing the eigenvalues and the dimension of the eigenspaces of the spectral decomposition \eqref{eq:shift_2}.
\end{proof}

What is interesting about the threshold $t_0=(N-1)^{-1}N\log N$ is that it does not depend on the metric structure of $(X,d)$ but rather on its topology, in particular the topological entropy of the left shift map, which as we already mentioned is $\log N$. It is also important to note that from the diagonalisation \eqref{eq:shift_2} it is evident that, up to a linear change of the eigenvalues, $\Ld$ coincides with the unbounded operator of Julien-Putnam in \cite{JP} for shift spaces. Although their operator is not intrinsic, it turns out their choice is natural as $\Ld$ depends only on $(X,d,\mu)$. In a subsequent paper we aim to study $\Ld$ in the more general setting of topological Markov chains and its compatibility with the dynamics of the shift map.

\subsection{Closed intervals}\label{sec:interval}

The metric-measure space in this subsection is the closed interval $[a,b]$ equipped with the Euclidean metric and the Lebesgue measure, which is clearly Ahlfors $1$-regular. 

An interesting first appearance of $\Ld$ can be traced back to the work of Tuck on slender-body potential theory in the 1960's \cite{T}. To be precise, the author considered neither the operator $\Ld$ nor Ahlfors regular spaces, but rather found an interesting relation between the Legendre polynomials $p_n$ on $[-1,1]$ and the harmonic numbers 
\begin{equation*}
h_n=\sum_{k=1}^n \frac{1}{k}.
\end{equation*}
Namely, Tuck proved that for every $n\in \mathbb N$ there is an equality
\begin{equation}\label{eq:interval_2}
\int_{-1}^{1}\frac{p_n(x)-p_n(y)}{|x-y|}\dd y = 2h_np_n(x).
\end{equation}
In our case, since the Legendre polynomials and the constant function $1$ form an orthonormal basis for $L^2([-1,1])$, equation \eqref{eq:interval_2} implies the following.

\begin{prop}[\cite{T}]
\label{prop: interval}
The logarithmic Dirichlet Laplacian $\Delta$ defined on the interval $[-1,1]$ satisfies $\Ld p_n=2h_np_n$ for every $n\in \mathbb N$. As a result, the Lipschitz functions on $[-1,1]$ are a core for $\Ld$ and the compact operator $e^{-t\Ld}$ is trace class if and only if $t>2^{-1}$.
\end{prop}

\begin{proof}
The relation $\Ld p_n=2h_np_n$ is proved for the interval $[-1,1]$ in \cite[Appendix II]{T}.
The summability threshold is implied by the fact that the sequence $h_n-\log n$ is convergent, with limit the Euler-Mascheroni constant.
\end{proof}
For an arbitrary interval $[a,b]$ the standard affine homeomorphism $$\varphi:[-1,1]\to [a,b],\quad t\mapsto \frac{(1-t)a+(t+1)b}{2},$$
induces a nonunitary invertible operator $T_{\varphi}:L^{2}[a,b]\to L^{2}[-1,1]$ that satisfies $\Delta_{[-1,1]}T_{\varphi}=T_{\varphi}\Delta_{[a,b]}$. This fact can be used to deduce similar statements for $\Delta_{[a,b]}$ and the shifted Legendre polynomials on $[a,b]$.

An intriguing question is whether Proposition \ref{prop: interval} can be generalised to Jacobi polynomials. The case of the interval is already remarkable though. For instance, our discussion shows that the restriction $\Delta:C^{2}([a,b])\to L^2([a,b])$ is essentially self-adjoint, while the ordinary Laplacian $-\frac{\mathrm{d}^{2}}{\mathrm{d}x^{2}}$ on $[a,b]$ is not essentially self-adjoint on $C^{2}[a,b]$, as $[a,b]$ has a boundary.

\subsection{Compact groups}\label{sec:groups}
In the present section we will discuss the logarithmic Dirichlet Laplacian on compact Ahlfors regular topological groups. We start with some slightly more general statements concerning self-maps of Ahlfors regular metric-measure spaces $(X,d,\mu)$.
\begin{definition}
A homeomorphism $\gamma:X\to X$ is an \emph{automorphism} if $\gamma$ preserves both the metric $d$ and measure $\mu$.
\end{definition}
For a Hilbert space $H$ we denote by $\mathcal{U}(H)$ its unitary group. An automorphism $\gamma:X\to X$ induces a unitary operator $U(\gamma)\in\mathcal{U}(L^{2}(X,\mu))$ via $U(\gamma)f(x)=f(\gamma^{-1}x)$. 
\begin{lemma}
\label{lem: group-commutator} 
Let $\gamma:X\to X$ be an automorphism and $U(\gamma)\in\mathcal{U}(L^{2}(X,\mu))$ the associated unitary operator. Then $U(\gamma):\Dom\Delta \to \Dom \Delta$ and $[\Delta,U(\gamma)]=0$.
\end{lemma}
\begin{proof} A straightforward change of variables shows that $$\mathcal{E}(f,f)=\mathcal{E}(U(\gamma)f,U(\gamma)f),\quad\mathcal{E}(U(\gamma)f,g)=\mathcal{E}(f,U(\gamma^{-1})g),$$ and the result follows from Proposition \ref{prop: domain-commutator-form}.
\end{proof}
Now we consider a locally compact topological group $G$ that acts by measure preserving isometries on $(X,d,\mu)$ and denote by $\nu$ the Haar measure on $G$. It is well known that $L^{1}(G,\nu)$ is a Banach algebra with the convolution product
\[f*g(t):=\int_{G}f(s)g(s^{-1}t)\mathrm{d}\nu(s).\] 
Moreover, $U:G\to \mathcal{U}(L^{2}(X,\mu))$ is a unitary representation of $G$ and $L^{2}(X,\mu)$ is a Banach module over $L^{1}(G,\nu)$ by setting
\begin{equation}
\label{eq: convolution-action}
U(f)\psi(x):=f*\psi=\int_{G}f(s)\psi(s^{-1}x)\mathrm{d}\nu(s)=\int_{G}f(s)U(s)\psi(x)\mathrm{d}\nu(s),\
\end{equation}
for all $f\in L^{1}(G,\nu),\psi\in L^{2}(X,\mu)$, see \cite[Proposition 2.1]{DeitmarEchteroff}.
\begin{cor} 
\label{cor:groups}
For $f\in L^{1}(G,\nu)$ and $g\in\Dom \Delta$ we have $\Delta(U(f)g)=U(f)\Delta g$.
\end{cor}
\begin{proof} This is immediate since $[\Delta,U(s)]=0$ for all $s\in G$.
\end{proof}
We now specialise to the case where $X=G$ is a compact group, $d$ is a bi-invariant metric and $\mu=\nu$ is the bi-invariant Haar measure. Some examples are compact Lie groups, compact groups admitting dilations \cite{Ma} and solenoid groups in hyperbolic dynamical systems \cite{Ge}. 

We also require the following folklore criterion for proving the essential self-adjointness of symmetric operators. A detailed statement and proof can, for instance, be found in \cite[Definition 2.2 and Proposition 2.6]{vdDungen}. The Lemma can be viewed as an abstract use of Friedrichs' mollifiers for symmetric differential operators, see \cite[Lemma 10.2.5]{HR}.

\begin{lemma}\label{lem:esa}
Let $\A:\Dom \A\to H$ be a symmetric operator that is densely defined on a Hilbert space $H$. Assume there is a sequence of bounded operators $(F_n)_{n\in \mathbb N}$ on $H$ such that
\begin{enumerate}
\item $\|F_n\|\lesssim 1$ and for every $f\in \Dom \A^{*}$ it holds that $\|F_nf-f\|_H\to 0$;
\item $F_n:\Dom \A^*\to \Dom \A$ for every $n\in \mathbb N$;
\item the commutator $[\A^*,F_n]$ extends to a bounded operator on $H$, with norm bounded independently of $n$.
\end{enumerate}
Then, the operator $\A$ is essentially self-adjoint.
\end{lemma} 

\begin{proof}
Since $\A$ is symmetric, the adjoint $\A^*$ extends the closure $\overline{\A}$. We claim that $\overline{\A}=\A^*$. Indeed, for every $f\in \Dom \A^*$ we can write $$\A F_nf=F_n\A^* f+[\A^* ,F_n]f.$$ Then, the norms $\|\A F_nf\|_H$ are bounded independently of $n\in \mathbb N$, while $\|F_nf-f\|_H\to 0$ and each $F_nf\in \Dom \A$. Consequently, from \cite[Lemma 1.8.1]{HR} it follows that $f\in \Dom \overline{\A}$.
\end{proof}

\begin{prop} Let $(G,d,\nu)$ be a compact Ahlfors regular metric group. Then, the restriction $\Delta_0:\Lip(G,d)\to L^{2}(G,\nu)$ is essentially self-adjoint.
\end{prop}

\begin{proof} We construct a sequence $\phi_{n}\in L^{1}(G,\nu)$ such that $F_{n}:=U(\phi_{n})$ satisfies the hypotheses of Lemma \ref{lem:esa}. Since $(G,d)$ is a metric space and a locally compact group, by \cite[Lemma 6.2.2 ]{DeitmarEchteroff} there exists a sequence $\psi_{n}\in L^{1}(G,\nu)$ such that $\|\psi_{n}\|_{L^{1}}=1$ and $\|U(\psi_{n})f-f\|_{L^{2}}\to 0$ for all $f\in L^{2}$. As $\Lip(G,d)\subset L^{1}(G,\nu)$ is a dense subspace we can chose $\phi_{n}\in\Lip(G,d)$ such that $\|\psi_{n}-\phi_{n}\|_{L^{1}}\leq n^{-1}$. Then
\[\|U(\phi_{n})\|\leq \|U(\psi_{n})\|+\|U(\psi_{n}-\phi_{n})\|\leq \|\psi_{n}\|_{L^{1}} + \|\psi_{n}-\phi_{n}\|_{L^{1}}\lesssim 1,\]
and 
\[\|U(\phi_{n})f-f\|\leq \|U(\psi_{n})f-f\|+\|U(\psi_{n}-\phi_{n})f\|\to 0.\]
Thus $F_{n}:=U(\phi_{n})$ still satisfies (1) of Lemma \ref{lem:esa}.
For (2) we observe that for $f\in\Lip(G,d)$ and $g\in L^{2}(G,\nu)$, right-invariance of the metric and unimodularity of $G$ give
\begin{align*}
|f*g(x)-f*g(y)|&=\left|\int_{G}(f(xs)-f(ys))g(s^{-1})\mathrm{d}s\right|\leq \Lip(f)\int_{G}d(xs,ys) g(s^{-1})\mathrm{d}s\\
&=\Lip(f)d(x,y)\int_{G} g(s^{-1})\mathrm{d}s=\Lip(f)d(x,y)\|g\|_{L^{1}},
\end{align*}
so that $f*g\in\Lip(G,d)$. In particular, $U(\phi_{n}):L^{2}(G,\nu)\to \Lip(G,d)$. Lastly, for (3), by left-invariance of the metric we can apply Lemma \ref{cor:groups} to find that $[\Delta^*,F_n]g=0$ for every $g\in \Dom \Delta^{*}$. 
Hence $\Delta_0$ is essentially self-adjoint on $\Lip(G,d)$.
\end{proof}
The product group $G\times G$ acts on $(G,d,\nu)$ by automorphisms via $(g_{1},g_{2})\cdot g:=g_{1}gg_{2}^{-1}$. The Lemma \ref{lem: group-commutator} together with Schur's Lemma imply that $\Delta_{\mathbb{C}}$ acts as a scalar on the irreducible summands in the Peter-Weyl decomposition of $L^{2}_{\mathbb{C}}(G,\nu)$. In particular, $\Delta_{\mathbb{C}}$ is essentially self-adjoint on the algebraic direct sum of the irreducible components in the Peter-Weyl decomposition of $L^{2}_{\mathbb C}(G,\nu)$. 

In the case of a compact Lie group $G$, $\Delta_{\mathbb{C}}$ is a bounded perturbation of the logarithm of the Casimir operator by \cite[Example 2.9]{GU}. For more general compact Ahlfors regular groups $\Delta_{\mathbb{C}}$ can thus be viewed as a Casimir-type operator that is diagonal on the Peter-Weyl decomposition.

The convolution \eqref{eq: convolution-action}, applied to $G\times G$ can be viewed as defining a left and a right Banach module structure of $L^{1}(G,\nu)$ on $L^{2}(G,\nu)$. Applying Corollary  \ref{cor:groups} then gives $$\Delta(f*g)=f*\Delta(g)=\Delta(f)*g,$$ for all $f\in\Lip(G,d)\subset L^{1}(G,\nu)$. In particular, $\Delta$ is an \emph{unbounded multiplier} of the Banach algebra $L^{1}(G,\nu)$.

For Ahlfors regular compact homogeneous spaces $(G/K,d,\mu)$, it is of interest to relate representation theoretic properties to the question of possible essential self-adjointness on $\Lip(G/K,d)$.

\subsection{Conformal boundary actions}\label{sec:CBA}
Let $\mathbb{B}^{n+1}\subset \mathbb{R}^{n+1}$ be the open unit ball and $\mathbb{S}^{n}$ its boundary, the $n$-sphere. Equipped with the Poincare metric, 
$$\mathrm{d}\rho=\frac{2\mathrm{d}x}{1-\|x\|^{2}},$$
where $\|x\|$ denotes the Euclidean norm on $\mathbb{R}^{n+1}$, the ball $\mathbb{B}^{n+1}$ is a model for hyperbolic $(n+1)$-space. The isometry group for the metric can be identified with $\mathrm{SO}^{+}(n+1,1)$ and the action of $\mathrm{SO}^{+}(n+1,1)$ extends to an action by diffeomorphisms on the boundary $\mathbb{S}^{n}$. 

We equip $\mathbb{S}^{n}$ with the structure of an Ahlfors regular metric-measure space of dimension $n$ by letting $\mu$ be the volume measure for the round metric on $\mathbb{S}^{n}$ and $d(x,y):=\|x-y\|$ the \emph{chordal distance} defined by the ambient Euclidean distance on $\mathbb{B}^{n+1}$. 

Each $\gamma\in \mathrm{SO}^{+}(n+1,1)$ defines a conformal diffeomorphism of $\mathbb{S}^{n}$. We denote by $\gamma':\mathbb{S}^{n}\to \mathrm{GL}_{n}(\mathbb{R})$ its derivative and by $|\gamma'(x)|$ the unique positive number for which $\frac{\gamma'(x)}{|\gamma'(x)|}$ is an orthogonal matrix. In particular $$|\gamma'|:\mathbb{S}^{n}\to \mathbb{R},\quad |\gamma'|:x\mapsto |\gamma'(x)|,$$ 
is a nonvanishing smooth map. We have the identities (see \cite[Equation 1.3.2]{Nicholls}),
\begin{equation}
\label{eq: derivative}
\|\gamma(x)-\gamma(y)\|=|\gamma'(x)|^{1/2}|\gamma'(y)|^{1/2}\|x-y\|,\quad |(\gamma\circ\delta)'(x)|=|\gamma'\circ\delta(x)||\delta'(x)|,
\end{equation}
and the measure $\mu$ satisfies (see \cite[Theorem 4.1.1]{Nicholls})
\begin{equation}
\label{eq: conformalfull}
\mathrm{d}\mu(\gamma(x))=|\gamma'(x)|^{n}\mathrm{d}\mu(x).
\end{equation}
Equations \eqref{eq: derivative} and \eqref{eq: conformalfull} imply that for any $\gamma\in\mathrm{SO}^{+}(n+1,1)$ and $F:X\times X\setminus D\to \mathbb{R}$ we have the equality
\begin{align}
\label{eq: change-of-variables}
\int_{\mathbb{S}^{n}}\int_{\mathbb{S}^{n}}\frac{F(x,y)}{d(x,y)^{n}}\mathrm{d}\mu(y)\mathrm{d}\mu(x)=\int_{\mathbb{S}^{n}}\int_{\mathbb{S}^{n}}\frac{F(\gamma x,\gamma y)}{d(x,y)^{n}}|\gamma'(x)|^{\frac{n}{2}}|\gamma'(y)|^{\frac{n}{2}}\mathrm{d}\mu(y)\mathrm{d}\mu(x).
\end{align}
On the Hilbert space $L^{2}(\mathbb{S}^{n},\mu)$ we consider the unitary representation
\begin{equation*}
U:\mathrm{SO}^{+}(n+1,1)\to \mathcal{U}(L^{2}(\mathbb{S}^{n},\mu)),\quad U(\gamma)f(x):=|(\gamma^{-1})'(x)|^{\frac{n}{2}}f(\gamma^{-1} x).
\end{equation*}
\begin{prop}\label{prop: conformalgroups}
Let $\Delta:\Dom\Delta\to L^{2}(\mathbb{S}^{n},\mu)$ be the logarithmic Dirichlet Laplacian for $(\mathbb{S}^{n},d,\mu)$ and $\gamma\in\mathrm{SO}^{+}(n+1,1)$. Then $U(\gamma)$ preserves $\Dom\Delta$ and the commutator $[\Delta,U(\gamma)]:\Dom \Delta\to L^2(\mathbb{S}^{n},\mu)$
extends to a bounded operator.
\end{prop}
\begin{proof} For notational convenience we prove the statement for $U(\gamma^{-1})$. We have
\begin{align*}
2\mathcal{E}(U(\gamma^{-1})f,U(\gamma^{-1})f)&=\int_{\mathbb{S}^{n}}\int_{\mathbb{S}^{n}}\frac{||\gamma'(x)|^{\frac{n}{2}}f(\gamma x)-|\gamma'(y)|^{\frac{n}{2}}f(\gamma y)|^{2}}{d(x,y)^{n}}\mathrm{d}\mu(x)\mathrm{d}\mu(y)\\
&\leq \int_{\mathbb{S}^{n}}\int_{\mathbb{S}^{n}}\frac{||\gamma'(x)|^{\frac{n}{2}}-|\gamma'(y)|^{\frac{n}{2}}|^{2}||f(\gamma x)|^{2}}{d(x,y)^{n}}\mathrm{d}\mu(x)\mathrm{d}\mu(y)\\ &\quad\quad\quad\quad\quad\quad\quad\quad\quad\quad+\int_{\mathbb{S}^{n}}\int_{\mathbb{S}^{n}}\frac{|\gamma'(y)|^{n}|f(\gamma x)-f(\gamma y)|^{2}}{d(x,y)^{n}}\mathrm{d}\mu(x)\mathrm{d}\mu(y),
\end{align*}
and we denote the integrals by $\I_1$ and $\I_2$. We estimate $\I_1$ using Lemma \ref{lem:Ahlfors_estimates}:
\begin{align*}
\I_1&=\int_{\mathbb{S}^{n}}\int_{\mathbb{S}^{n}}\frac{||\gamma'(x)|^{\frac{n}{2}}-|\gamma'(y)|^{\frac{n}{2}}||^{2}}{d(x,y)^{n}}\mathrm{d}\mu(y)|f(\gamma x)|^{2}\mathrm{d}\mu(x)\\
&\leq \Lip(|\gamma'|^{\frac{n}{2}})^{2}\int_{\mathbb{S}^{n}}\int_{\mathbb{S}^{n}}\frac{1}{d(x,y)^{n-2}}\mathrm{d}\mu(y)|f(\gamma x)|^{2}\mathrm{d}\mu(x)\\
&\lesssim \Lip(|\gamma'|^{\frac{n}{2}})^{2}\int_{\mathbb{S}^{n}}|f(\gamma x)|^{2}\mathrm{d}\mu(x)\\
&\lesssim \Lip(|\gamma'|^{\frac{n}{2}})^{2}\||(\gamma^{-1})'|^n\|_{\infty}\|f\|_{2}^{2}.
\end{align*}
We rewrite $\I_2$ using \eqref{eq: change-of-variables} and the chain rule and estimate
\begin{align*}
\I_{2}&=\int_{\mathbb{S}^{n}}\int_{\mathbb{S}^{n}}\frac{|f(x)-f( y)|^{2}}{d(x,y)^{n}}\left(\frac{|(\gamma^{-1})'(x)|}{|(\gamma^{-1})'(y)|}\right)^{\frac{n}{2}}\mathrm{d}\mu(x)\mathrm{d}\mu(y)\leq \left\|\left(\frac{|(\gamma^{-1})'(x)|}{|(\gamma^{-1})'(y)|}\right)^{\frac{n}{2}}\right\|_{\infty}2\mathcal{E}(f,f).
\end{align*}
Hence, $U(\gamma^{-1}):\W\to\W$. A straightforward calculation using \eqref{eq: change-of-variables} shows that
\begin{align*}
|\mathcal{E}(U(\gamma^{-1})f,g)-\mathcal{E}(f,U(\gamma)g)|&=\left|\int_{\mathbb{S}^{n}}\int_{\mathbb{S}^{n}}\frac{|\gamma'(y)|^{n/2}-|\gamma'(x)|^{n/2}}{d(x,y)^{n}}f(\gamma x)g(x)\mathrm{d}\mu(y)\mathrm{d}\mu(x)\right|\\
&\lesssim\Lip(|\gamma'|^{\frac{n}{2}})\int_{\mathbb{S}^{n}}\int_{\mathbb{S}^{n}}\frac{1}{d(x,y)^{n-1}}\mathrm{d}\mu(y)|f(\gamma x)||g(x)|\mathrm{d}\mu(x)\\
&\lesssim \Lip(|\gamma'|^{\frac{n}{2}})\||(\gamma^{-1})'|^{\frac{n}{2}}\|_{\infty}\|f\|_{L^{2}}\|g\|_{L^{2}}.
\end{align*}
Thus by Proposition \ref{prop: domain-commutator-form}, $U(\gamma^{-1})$ preserves $\Dom\Delta$ and the commutator $[\Delta,U(\gamma^{-1})]$ extends to a bounded operator.
\end{proof}
We now consider a discrete subgroup $\Gamma\subset \mathrm{SO}^{+}(n+1,1)$, which we refer to as a \emph{Kleinian group}. The \emph{limit set} $\Lambda_{\Gamma}\subset \mathbb{S}^{n}$ is the set of accumulation points of an orbit $\Gamma x\subset\mathbb{B}^{n+1}$, that is $\Lambda_{\Gamma}=\overline{\Gamma x}\cap \mathbb{S}^{n}$. This definition is independent of the choice of $x\in\mathbb{B}^{n+1}$. If $\Lambda_{\Gamma}=\mathbb{S}^{n}$ then $\Gamma$ is said to be of the \emph{first kind} and otherwise it is of the \emph{second kind}. 

The \emph{convex core} $\mathcal{C}(\Gamma)\subset\mathbb{B}^{n+1}$ of such a group is the union of all hyperbolic geodesics both whose endpoints are elements of $\Lambda_{\Gamma}$. The group $\Gamma$ is \emph{convex cocompact} if $\mathcal{C}(\Gamma)/\Gamma$ is compact. If $\Gamma$ is of the first kind then $\mathcal{C}(\Gamma)=\mathbb{B}^{n+1}$ and it is convex cocompact if and only it is cocompact.

The \emph{Patterson-Sullivan measure} on $\Lambda_{\Gamma}$ is a measure on the limit set \cite{Nicholls, Su}, which we denote by $\mu$. If $\Gamma$ is convex cocompact, by \cite[Theorem 4.1.1]{Nicholls}, $\mu$ satisfies the transformation rule
\begin{equation}
\label{eq: conformalHausdorff}
\mathrm{d}\mu(\gamma(x))=|\gamma'(x)|^{\delta}\mathrm{d}\mu(x),
\end{equation}
where $\delta$ is the Hausdorff dimension of $\Lambda_{\Gamma}$ with respect to the chordal distance $d$ defined above. The group $\Gamma$ acts on $\Lambda_{\Gamma}$ and we have a unitary representation
\[U:\Gamma\to\mathcal{U}(L^{2}(\Lambda_{\Gamma},\mu)),\quad U(\gamma)f(x):=|(\gamma^{-1})'(x)|^{\frac{\delta}{2}}f(\gamma^{-1} x).\]
Moreover, by \cite[Theorem 4.6.2]{Nicholls}, $(\Lambda_{\Gamma}, d,\mu)$ is an Ahlfors regular metric-measure space. Combining \eqref{eq: conformalHausdorff} with \eqref{eq: derivative}, we obtain the analogue of \eqref{eq: change-of-variables} with $n$ replaced by $\delta$. The same proof as that of Proposition \ref{prop: conformalgroups} then shows the following.
\begin{prop}\label{prop:CBA2}
Let $\Gamma\subset\mathrm{SO}^{+}(n+1,1)$ be a convex cocompact Kleinian group of the second kind, $\Delta:\Dom\Delta\to L^{2}(\Lambda_{\Gamma},\mu)$ the logarithmic Dirichlet Laplacian for $(\Lambda_{\Gamma},d,\mu)$ and $\gamma\in\Gamma$. Then $U(\gamma)$ preserves $\Dom\Delta$ and the commutator $[\Delta,U(\gamma)]:\Dom \Delta\to L^2(\Lambda_{\Gamma},\mu)$
extends to a bounded operator.
\end{prop}
\begin{remark}\label{rem:CBA}
We conclude by emphasizing that an analogue of Propositions \ref{prop: conformalgroups} and \ref{prop:CBA2} cannot be obtained for the fractional Dirichlet Laplacians $\Delta_{\alpha}$ described in the introduction. 
We consider the complex operators $\Delta_{\mathbb C}, \Delta_{\alpha,\mathbb C}$. The operator $\Delta_{\alpha,\mathbb C}$ on $(\mathbb S^n,d,\mu)$ has been studied by Samko \cite{Samko, Samko2} and can be diagonalised by the spherical harmonics and his proof shows that $\Delta_{\alpha}$ is finitely summable in the sense of Definition \ref{def: summable}. Using Proposition \ref{prop: domain-commutator-form} of the present paper, we thus obtain a finitely summable spectral triple $(C_{\mathbb C}^{\infty}(\mathbb S^n), L_{\mathbb C}^2(\mathbb S^n,\mu),\Delta_{\alpha,\mathbb C})$.

We denote the algebra generated by $C^{\infty}_{\mathbb{C}}(\mathbb{S}^{n})\subset \mathbb{B}(L^{2}_{\mathbb{C}}(\mathbb{S}^{n},\mu))$ and $U(\Gamma)$ by $C_{\mathbb C}^{\infty}(\mathbb S^n)\rtimes_{\text{alg}} \Gamma$. Its norm closure is the \emph{crossed product} $C(\mathbb{S}^{n})\rtimes \Gamma$ (see \cite{W}). Now, if Proposition \ref{prop: conformalgroups} were true also for $\Delta_{\alpha}$ and so for $\Delta_{\alpha,\mathbb C}$, we would get a finitely summable spectral triple $(C_{\mathbb C}^{\infty}(\mathbb S^n)\rtimes_{\text{alg}} \Gamma,L_{\mathbb C}^2(\mathbb S^n,\mu),\Delta_{\alpha,\mathbb C})$ for any discrete subgroup $\Gamma$ of $\mathrm{SO}^{+}(n+1,1)$. Then, following Connes \cite[Theorem 8]{Connestrace} one obtains a trace on the crossed product $C^*$-algebra $C(\mathbb S^n)\rtimes \Gamma$. However, if $\Gamma$ is non-elementary and of the first kind, then $\Lambda_{\Gamma}=\mathbb{S}^{n}$ and $C(\mathbb S^n)\rtimes \Gamma$ does not admit any trace, as it is a purely infinite $C^*$-algebra, see \cite[Proposition 3.1]{Del} and \cite[Lemma 3.8]{Lott}, leading to a contradiction. 

The weaker $\Li$-summability of the spectral triples $(\D_{\mathbb{C}}(\mathbb S^n,d)\rtimes_{\text{alg}} \Gamma, L^{2}_{\mathbb{C}}(\mathbb S^n,\mu), \Delta_{\mathbb{C}}),$ associated to any of the aforementioned crossed products $C(\mathbb S^n)\rtimes \Gamma$ circumvents Connes' tracial obstruction, and as such is an asset. We leave the study of spectral triples constructed from $\Delta$ to a subsequent paper.
\end{remark}

\begin{remark}
The results of the present subsection show that we obtain a representation of $\Gamma$ on $\Dom \Delta$ as bounded operators. In the recent work \cite{BS} the authors consider uniformly bounded representations of hyperbolic groups on \textit{fractional} Sobolev spaces. Although our context differs, it would be of interest to investigate the relationship of their work with ours.
\end{remark}

\begin{Acknowledgements}
The authors would like to express their gratitude to Kevin Boucher, Magnus Goffeng, Erik Koelink, Franz Luef and J\'{a}n $\check{\text{S}}$pakula for many stimulating discussions on the subject of this paper. We also thank the referees for their careful reading of the manuscript.
\end{Acknowledgements}

\end{document}